 \documentclass[reqno,10pt]{amsart}

\usepackage{amsmath, amsthm, amscd, amsfonts, amssymb, graphicx, color,mathrsfs,graphicx,extpfeil, a4wide}
\usepackage[bookmarksnumbered, colorlinks, plainpages]{hyperref}
\usepackage{tensor}
\usepackage{bm}

\usepackage{epsfig}
\usepackage{indentfirst, latexsym, amssymb, enumerate,amsmath,graphicx}
\usepackage{float}
\usepackage{colortbl}
\usepackage{epsfig,subfigure}
\usepackage{bm}
\usepackage{empheq}
\usepackage{color}
\usepackage{setspace}
\usepackage{bbm}
\usepackage{lipsum}
\usepackage[capitalize,nameinlink,noabbrev]{cleveref}
\usepackage{mathrsfs}
\usepackage{cleveref}
\usepackage{tikz}
\usetikzlibrary {patterns}
\makeatletter

\def\section{\@startsection{section}{1}%
  \z@{.7\linespacing\@plus\linespacing}{.5\linespacing}%
  {\normalfont\large\bfseries\centering}}
\def\subsubsection{\@startsection{subsubsection}{1}%
  \z@{.7\linespacing\@plus\linespacing}{.5\linespacing}%
  {\normalfont}}
\renewenvironment{abstract}{%
  \ifx\maketitle\relax
    \ClassWarning{\@classname}{Abstract should precede
      \protect\maketitle\space in AMS document classes; reported}%
  \fi
  \global\setbox\abstractbox=\vtop \bgroup
    \normalfont\Small
    \list{}{\labelwidth\z@
      \leftmargin3pc \rightmargin\leftmargin
      \listparindent\normalparindent \itemindent\z@
      \parsep\z@ \@plus\p@
      
    }%
    \item[\hskip\labelsep\bfseries\abstractname.]%
}{%
  \endlist\egroup
  \ifx\@setabstract\relax \@setabstracta \fi
}
\def\@setabstract{\@setabstracta \global\let\@setabstract\relax}
\def\@setabstracta{%
  \ifvoid\abstractbox
  \else
    \skip@0\p@ \advance\skip@-\lastskip
    \advance\skip@-\baselineskip \vskip\skip@
    \box\abstractbox
    \prevdepth\z@ 
  \fi
}

\renewcommand{\tocsection}[3]{%
  \indentlabel{\@ifnotempty{#2}{\bfseries\ignorespaces#1 #2\quad}}\bfseries#3}
\renewcommand{\tocsubsection}[3]{%
  \indentlabel{\@ifnotempty{#2}{\ignorespaces#1 #2\quad}}#3}

\newcommand\@dotsep{4.5}
\def\@tocline#1#2#3#4#5#6#7{\relax
  \ifnum #1>\c@tocdepth 
  \else
    \par \addpenalty\@secpenalty\addvspace{#2}%
    \begingroup \hyphenpenalty\@M
    \@ifempty{#4}{%
      \@tempdima\csname r@tocindent\number#1\endcsname\relax
    }{%
      \@tempdima#4\relax
    }%
    \parindent\z@ \leftskip#3\relax \advance\leftskip\@tempdima\relax
    \rightskip\@pnumwidth plus1em \parfillskip-\@pnumwidth
    #5\leavevmode\hskip-\@tempdima{#6}\nobreak
    \leaders\hbox{$\m@th\mkern \@dotsep mu\hbox{.}\mkern \@dotsep mu$}\hfill
    \nobreak
    \hbox to\@pnumwidth{\@tocpagenum{\ifnum#1=1\bfseries\fi#7}}\par
    \nobreak
    \endgroup
  \fi}
\AtBeginDocument{%
\expandafter\renewcommand\csname r@tocindent0\endcsname{0pt}
}
\def\l@subsection{\@tocline{2}{0pt}{2.5pc}{5pc}{}}
\def\l@subsubsection{\@tocline{2}{0pt}{4pc}{5pc}{}}

\makeatother

\makeatletter
\newsavebox{\@brx}
\newcommand{\llangle}[1][]{\savebox{\@brx}{\(\m@th{#1\langle}\)}%
	\mathopen{\copy\@brx\kern-0.5\wd\@brx\usebox{\@brx}}}
\newcommand{\rrangle}[1][]{\savebox{\@brx}{\(\m@th{#1\rangle}\)}%
	\mathclose{\copy\@brx\kern-0.5\wd\@brx\usebox{\@brx}}}
\makeatother

\usepackage[width=122mm,left=12mm,paperwidth=146mm,height=193mm,top=12mm,paperheight=217mm]{geometry}

\newtheorem{thm}{Theorem}[section]
\newtheorem{prop}[thm]{Proposition}
\newtheorem{cor}[thm]{Corollary}
\newtheorem{lem}[thm]{Lemma}
\newtheorem{rem}[thm]{Remark}
\newtheorem{defn}[thm]{Definition}

\newtheorem*{theorem*}{Theorem}

\numberwithin{equation}{section}

\newcommand{\tR}{{R^s }}
\newcommand{\btR}{{\underline{R}^s }}
\newcommand{\ric}{{\rm Ric}}

\newcommand{\lR}{\underline{R}}
\newcommand{\lchi}{\underline{\chi}}
\newcommand{\ddbar}{\partial\overline{\partial}}

\newcommand{\tr}{{\mathrm{tr}}}

\newcommand{\bL}{{\mathbb L}}

\newcommand{\bR}{{\mathbb R}}

\newcommand{\mC}{{\mathcal C}}
\newcommand{\mD}{{\mathfrak{D}}}
\newcommand{\mE}{{\mathcal E}}

\newcommand{\mH}{{\mathcal H}}

\newcommand{\mM}{{\mathcal M}}

\newcommand{\mV}{{\mathcal V}}

\title{Twisted Calabi functional and twisted Calabi flow}
\author{Jie He}
\address{School of Mathematics and Physics, Beijing University of Chemical Technology, Beijing 10016}
\email{hejie@amss.ac.cn}

\author{Haozhao Li}
\address{Institute of Geometry and Physics, and Key Laboratory of Wu Wen-Tsun Mathematics, School of Mathematical Sciences, University of Science and Technology of China}
\email{hzli@ustc.edu.cn}

\thanks{
	Jie He is supported by Mathematics Tianyuan Fund of National Natural Science Foundation of China   No. 12426674.
	Haozhao Li is supported by NSFC grant No. 12426669, No. 12471058,  the CAS Project for Young Scientists
	in Basic Research (YSBR-001), and the Fundamental Research Funds
	for the Central Universities.}

\newtheorem{theorem}{Theorem}[section]
\newtheorem{lemma}{Lemma}[section]

\newtheorem{remark}{Remark}[section]

\newtheorem{definition}{Definition}[section]

\begin{document}




\maketitle

\begin{abstract}
	This paper investigates the twisted Calabi functional and the associated twisted Calabi flow on compact K\"ahler manifolds. Our main contributions are threefold: first, we establish the convexity of the twisted Calabi functional at its critical points; second, we prove the short-time existence of the twisted Calabi flow; and third, we demonstrate the stability of this flow in the neighborhood of twisted constant scalar curvature K\"ahler metrics. These results provide an   analytic foundation for studying the twisted Calabi flow and  resolve questions about its local behavior.
\vskip 2.5mm
\noindent {\bf Keywords.} twisted Calabi flow, canonical K\"ahler metrics, geometric flow.

\vskip 2.5mm
\noindent {\bf MSC 2020.} Primary 53C21 ; Secondary 53C55, 58J05, 58J60 .
\end{abstract}

\tableofcontents

\section{Introduction}
Twisted scalar curvature serves as a natural generalization of the K\"ahler scalar curvature, introduced to address the existence problem of constant scalar curvature K\"ahler (cscK) metrics on compact K\"ahler manifolds. The notion of twisted scalar curvature originated in the work of J. Fine on cscK metrics over fibered surfaces (see \cite{MR2144537, MR2318622}), as well as in the study of the K\"ahler-Ricci flow on K\"ahler surfaces by Song and Tian (see \cite{MR2357504}). In 2009, J. Stoppa (see \cite{MR2581360}) provided a moment map interpretation of the twisted scalar curvature and established a slope stability criterion for the existence of twisted cscK metrics. In 2015, R. Dervan proved uniform stability for such metrics in \cite{MR3564626}.

In 2015, X. X.  Chen redefined a more general notion of twisted cscK metrics in \cite{MR3858468}, where he treated them as forming a continuity path connecting the cscK equation and the $J$-equation, and proposed a comprehensive program for studying classical cscK metrics through their twisted counterparts.
 The openness (\cite{MR3858468, MR3980270, MR4136486}) and closedness (\cite{MR4301557}) of this path ultimately led to the resolution of several long-standing conjectures, including Donaldson's geodesic stability conjecture and the coercivity conjecture for cscK metrics \cite{MR4301557, MR4301558, chen2018constant}. This continuity path was also employed by Chen, P\u{a}un, and Zeng to investigate the uniqueness of extremal and cscK metrics in \cite{chen2015deformation}.

The twisted Calabi flow is the gradient flow of the twisted K-energy functional associated with the twisted scalar curvature. Within his framework for twisted cscK metrics \cite{MR3858468}, Chen proposed the study of this flow. In the case of Riemann surfaces, J. Pook in \cite{MR3514059} established its long-time existence and convergence. The aim of this paper is to investigate the twisted Calabi flow and related problems on compact K\"ahler manifolds of arbitrary dimensions.

\subsection{Some notations}
To state our results, we first introduce some basic notations. Let $(M, \omega)$ be a compact K\"ahler manifold of complex dimension $m$. For another K\"ahler metric $\chi$ on $M$ and a parameter $s \in [0,1]$, the $(\chi, s)$-twisted Ricci curvature is defined as
\[
\mathrm{Ric}^s := s \mathrm{Ric} - (1 - s)\chi.
\]
The corresponding twisted scalar curvature is given by
\[
R^s := s R(\omega) - (1 - s) \tr_\omega \chi.
\]
A K\"ahler metric $\omega$ is called a twisted cscK metric if it satisfies
\begin{align}\label{eq:twisted-csck}
	s R(\omega) - (1 - s) \tr_\omega \chi = s \underline{R} - (1 - s) \underline{\chi},
\end{align}
where $\underline{R}$ and $\underline{\chi}$ denote the averages of $R(\omega)$ and $\tr_\omega \chi$ over $M$, respectively. Clearly, twisted cscK metrics generalize classical cscK metrics.

The twisted Calabi functional is defined as the $L^2$-norm of the deviation of the twisted scalar curvature from its average:
\[
\mC^s(\omega) = \int_M \left( s R(\omega) - (1 - s) \tr_\omega \chi - s \underline{R} + (1 - s) \underline{\chi} \right)^2 \omega^m.
\]For $s \in [0,1]$, the twisted K-energy functional is given by
\[
\mM^s(\varphi) = s \mM(\varphi) + (1 - s) J_\chi(\varphi), \quad \varphi \in \mH,
\]
where $\mH = \{\varphi \in C^\infty(M, \mathbb{R}) : \omega_\varphi := \omega + i\partial\bar\partial \varphi > 0 \}$ is the space of K\"ahler potentials, $\mM$ is the Mabuchi K-energy \cite{MR0796483}, and $J_\chi$ is the $J$-functional \cite{MR1772078}. The twisted cscK equation is the Euler-Lagrange equation for $\mM^s$, and the twisted Calabi flow is its gradient flow:
\[
\frac{\partial \varphi}{\partial t} = R^s(\varphi) - \underline{R}^s,
\]
where $R_\varphi$ is the scalar curvature of $\omega_\varphi$, $\tr_\varphi \chi = \tr_{\omega_\varphi} \chi$, and $\underline{R}^s = s \underline{R} - (1 - s) \underline{\chi}$.

\begin{defn}
	The twisted Lichnerowicz operator $\bL^s$ is defined by
	\[
	\bL^s(f) = s \mD^*\mD f + (1 - s) i \bar\partial^* (\nabla^{1,0} f \lrcorner \chi), \quad \forall f \in C^\infty(M, \mathbb{R}),
	\]
	where $\mD f = \bar\partial \nabla^{1,0} f$, and $\mD^*, \bar\partial^*$ denote the formal adjoints of $\mD$ and $\bar\partial$, respectively.
\end{defn}

For $s \in(0,1)$, the operator $\bL^s$ is a fourth-order semi-positive elliptic operator with $\ker \bL^s = \{\text{constants}\}$. It generalizes the classical Lichnerowicz operator (see \cref{Classical Lichnerowicz operator}) and plays a central role in our analysis.

\subsection{Main results}
We now present our main results. First, we compute the first variation of the twisted Calabi functional.

\begin{thm}[\cref{thm:first-variation}]
	Assume $s \in (0,1)$. The first variation of the twisted Calabi functional is given by
	\[
	D(\mC^s)_\omega(\varphi) = - \int_M \langle \varphi, \bL^s(R^s) \rangle \, \omega^m,
	\]
	where $\bL^s$ is the twisted Lichnerowicz operator. Consequently, $\omega$ is a critical point of $\mC^s$ if and only if it is a twisted cscK metric.
\end{thm}

\begin{rem}
	In \cite{MR3858468}, Chen introduced the notion of twisted extremal K\"ahler metrics, defined by the condition that $\nabla^{1,0}(s R_\omega - (1 - s) \tr_\omega \chi)$ is a holomorphic vector field. However, such metrics do not arise as critical points of the twisted Calabi functional, in contrast to the classical Calabi functional.
\end{rem}

As a consequence of the first variation formula, the twisted Calabi flow decreases the Mabuchi distance in $\mH$.

\begin{cor}
	Let $\varphi(\tau) : [0,1] \to \mH$ be a smooth curve, and let $\varphi(t, \tau)$ be its deformation under the twisted Calabi flow. Denote by $l(t)$ the length of the curve $\varphi_t(\tau) = \varphi(t, \tau)$. Then
	\[
	\frac{d l(t)}{d t} = - \int_0^1 \left( \int_M \left( \frac{\partial \varphi}{\partial \tau} \right)^2 \omega_\varphi^m \right)^{-\frac{1}{2}} \left( \int_M \frac{\partial \varphi}{\partial \tau} \bL^s \left( \frac{\partial \varphi}{\partial \tau} \right) \omega_\varphi^m \right) d\tau \leq 0.
	\]
	Hence, the twisted Calabi flow strictly decreases the distance in $\mH$ unless the curve degenerates to a point.
\end{cor}

Another application of the variation result is to prove the convexity of twisted Mabuchi Energy along geodesics(see \cref{thm:convexity-K-energy}), which has been proved by Berman-Berndsston in \cite{MR3671939} along weak geodescis and Berman-Darvas-Lu(\cite{MR3687111})
in the space
$\mE^p$.

Similar to the classical case, the twisted Calabi functional is convex near its critical points.

\begin{thm}
	Assume $s \in (0,1)$. At a critical point of $\mC^s$, the Hessian is given by
	\[
	\mathrm{Hess}\, \mC^s(\varphi, \phi) = \int_M \langle \bL^s(\varphi), \bL^s(\phi) \rangle \, \omega^m, \quad \forall \varphi, \phi \in C^\infty(M, \mathbb{R}).
	\]
	As consequences:
	\begin{enumerate}
		\item $\mathrm{Hess}\, \mC^s$ is strictly positive in $\mH/\{constants\}$ at a critical point;
		\item Twisted cscK metrics are isolated in a K\"ahler class.
	\end{enumerate}
\end{thm}

In a fixed K\"ahler class, cscK metrics (and extremal K\"ahler metrics) form an orbit under the identity component of the holomorphic automorphism group, so uniqueness holds only modulo this group action. In contrast, twisted cscK metrics are isolated due to the convexity result above.
The isolation of twisted cscK metrics is also a consequence of R. J. Berman and B. Berndtsson in \cite{MR3671939}, who proved there exists at most one twisted cscK metric in a K\"ahler class.

We now turn to the twisted Calabi flow. Short-time existence and stability near cscK metrics for the classical Calabi flow were established by Chen-He \cite{MR2405167}. Later, He \cite{MR3010280} and He-Zeng \cite{MR4259154} relaxed the regularity requirements on initial data. For the twisted Calabi flow on Riemann surfaces, Pook proved existence and convergence in \cite{MR3514059}. Following the approaches in \cite{MR2405167, MR3010280}, we establish short-time existence for the twisted Calabi flow in arbitrary dimension.

We first introduce some function spaces (see \cref{sec:short-time-existence} for precise definitions). For a Banach space $E$ and $T > 0$, define
\[
C_{1/2}([0,T], E) = \left\{ u \in C((0,T], E) \;\middle|\; t \mapsto t^{1/2} u(t) \in C([0,T], E),\ \lim_{t \to 0} t^{1/2} u = 0 \right\}.
\]
Let $c^{k,\alpha}(M)$ denote the closure of $C^\infty(M)$ in $C^{k,\alpha}(M)$ under the $\| \cdot \|_{C^{k,\alpha}}$ norm. Define
\[
E_0([0,T]) = C_{1/2}([0,T], c^\alpha(M)), \quad \| u \|_{E_0([0,T])} = \sup_{t \in [0,T]} \| t^{1/2} u(t) \|_{c^\alpha(M)},
\]
and
\[
E_1([0,T]) = \left\{ u \in C_{1/2}([0,T], c^{4,\alpha}(M)) \;\middle|\; \dot u \in E_0([0,T]) \right\},
\]
with norm
\[
\| u \|_{E_1([0,T])} = \sup_{t \in [0,T]} t^{1/2} \left( \| \dot u(t) \|_{c^\alpha(M)} + \| u(t) \|_{c^{4,\alpha}(M)} \right).
\]

\begin{thm}\label{t1.7}
	For any smooth K\"ahler metrics $\omega, \chi$ on $M$ and $s \in (0,1]$, there exist constants $T =  C(\omega, \chi) s^2$, $\epsilon = \epsilon(\omega, \chi) > 0$, and $c = c(\omega) > 0$ such that for any initial value $x \in c^{2,\alpha}(M)$ with $\| x \|_{C^{2,\alpha}(M)} \leq \epsilon$, the twisted Calabi flow equation
	\[
	\begin{cases}
		\partial_t \varphi = R^s(\omega_\varphi) - \underline{R}^s, \\
		\varphi(0) = x
	\end{cases}
	\]
	has a unique solution $\varphi(t, x) \in E_1([0,T])$. Moreover, $\varphi \in C([0,T], c^{2,\alpha}(M))$, and for any $x, y \in B_{c^{2,\alpha}(M)}(0, \epsilon)$,
	\begin{align}
		\|\varphi(t, x) - \varphi(t, y)\|_{C([0,T], c^{2,\alpha}(M))} &\leq c \| x - y \|_{c^{2,\alpha}(M)}, \label{eq:lip-cont} \\
		\|\varphi(t, x) - \varphi(t, y)\|_{E_1([0,T])} &\leq c \| x - y \|_{c^{2,\alpha}(M)}. \label{eq:E1-bound}
	\end{align}
\end{thm}

Another basic question, raised by Chen in \cite{MR3858468}, is whether the twisted Calabi flow is stable near twisted cscK metrics. The following theorem gives an affirmative answer.

\begin{thm}\label{thm:stability}
	Assume $p\geq 6m+2, s \in (0,1)$ and $\omega$ is a twisted cscK metric. There exists $\delta = C(\omega, \chi)s^p(1-s) > 0$ such that if $\| \varphi_0 \|_{c^{2,\alpha}(M)} < \delta$, then the twisted Calabi flow starting from $\varphi_0$ exists for all time and converges smoothly  to $0$. Moreover, $\|\varphi(t)\|_{C^{k,\alpha}(M)}$ converges to 0 exponentially as $t\to\infty$ for any $k\geq 1$.
\end{thm}

\begin{rem}
	The stability theorem for the classical Calabi flow \cite{MR2405167} states that if the initial metric is sufficiently close to a cscK metric $\omega_0$, the flow converges to some cscK metric, which may differ from $\omega_0$ by an automorphism of $(M, J)$. In contrast, by the uniqueness result of Berman-Berndtsson \cite{MR3671939}, the limiting twisted cscK metric in \cref{thm:stability} must be the original metric $\omega$.
\end{rem}

\section{ Preliminaries}
\label{sec:preliminaries}

\subsection{Notations and basic conventions}

Throughout this paper, $(M, \omega, J)$ denotes a compact K\"ahler manifold of complex dimension $m$ with K\"ahler form $\omega$. It's real dimension is $n=2m$. For any smooth function $f \in C^1(M)$, we define the $(1,0)$ and $(0,1)$ components of its gradient vector field as follows:
\begin{align*}
    \nabla^{1,0} f &= g^{\alpha\bar{\beta}} \frac{\partial f}{\partial \bar{z}_\beta} \frac{\partial}{\partial z_\alpha}, \\
    \nabla^{0,1} f &= g^{\alpha\bar{\beta}} \frac{\partial f}{\partial z_\alpha} \frac{\partial}{\partial \bar{z}_\beta}.
\end{align*}
The Riemannian metric $g$ extends $\mathbb{C}$-bilinearly to the complexified tangent bundle $T^\mathbb{C} M$. For complex-valued differential forms $\psi \in \Omega^{p,q}(M)$ and $\phi \in \Omega^{p,q}(M)$, we employ the Hermitian inner product:
\[
\langle \psi, \phi \rangle = g(\psi, \overline{\phi}).
\]
Let $\bar{\partial}^*$ denote the formal adjoint of $\bar{\partial}$ with respect to this inner product, characterized by the relation:
\[
\int_M \langle \bar{\partial}^* \psi, \phi \rangle \,\omega^m = \int_M \langle \psi, \bar{\partial} \phi \rangle \, \omega^m, \quad \forall \psi \in \Omega^{p,q+1}(M), \phi \in \Omega^{p,q}(M).
\]
The operator $\partial^*$ is defined analogously as the formal adjoint of $\partial$.

 We first introduce some basic formulas.

\begin{lem}\label{lem:tensoridentity}
    For any closed real $(1,1)$-form $\chi = i\chi_{\alpha\bar{\beta}} dz_\alpha \wedge d\bar{z}_\beta$ and smooth function $\varphi \in C^\infty(M)$, the following identity holds:
    \begin{align} \label{equ:deltagrad}
        i\bar{\partial}^* (\nabla^{1,0} \varphi \lrcorner \chi) = \langle i\ddbar \varphi, \chi \rangle + \langle \partial \tr_\omega \chi,  \partial  \varphi \rangle.
    \end{align}
\end{lem}

\begin{proof}
    The closedness of $\chi$ implies the commutation relation $\nabla_\alpha \chi_{\xi\bar{\beta}} = \nabla_\xi \chi_{\alpha\bar{\beta}}$. We compute directly:
    \begin{align*}
        i\bar{\partial}^* (\nabla^{1,0} \varphi \lrcorner \chi)
        &= g^{\alpha\bar{\beta}} \nabla_\alpha \left( g^{\gamma\bar{\eta}} \varphi_{\bar{\eta}} \chi_{\gamma\bar{\beta}} \right) \\
        &= g^{\alpha\bar{\beta}} g^{\gamma\bar{\eta}} \varphi_{\alpha\bar{\eta}} \chi_{\gamma\bar{\beta}} + g^{\alpha\bar{\beta}} g^{\gamma\bar{\eta}} \varphi_{\bar{\eta}} \nabla_\alpha \chi_{\gamma\bar{\beta}} \\
        &= \langle i\ddbar \varphi, \chi \rangle + g^{\alpha\bar{\beta}} g^{\gamma\bar{\eta}} \varphi_{\bar{\eta}} \nabla_\gamma \chi_{\alpha\bar{\beta}} \\
        &= \langle i\ddbar \varphi, \chi \rangle + \langle \partial \tr_\omega \chi,  \partial \varphi \rangle.
    \end{align*}
    Taking the complex conjugate yields the  identity:
    \[
    i\partial^* (\nabla^{0,1} \varphi \lrcorner \chi) = -\langle i\ddbar \varphi, \chi \rangle - \langle \bar{\partial} \tr_\omega \chi, \bar\partial \varphi \rangle.
    \]
\end{proof}

\subsection{Twisted canonical metrics}

We now introduce the twisted Ricci curvature and twisted scalar curvature.

\begin{defn}[Twisted Ricci curvature]
    For a K\"ahler form $\chi$ and parameter $s \in [0,1]$, the \emph{twisted Ricci curvature} is defined by:
    \[
    \ric^s = s \ric - (1-s) \chi,
    \]
    and the corresponding \emph{twisted scalar curvature} is given by:
    \[
    R^s = -(1-s) \tr_\omega \chi + s R(\omega).
    \]
\end{defn}

Unless otherwise specified, we maintain the assumption $s \in (0,1)$ throughout this work.

\begin{defn}[Twisted cscK metric]
    A K\"ahler metric $\omega$ is called a \emph{twisted constant scalar curvature K\"ahler (cscK) metric} if there exists a constant $C^s$ such that:
    \begin{align}\label{twisted cscK equaiton}
    s R(\omega) - (1-s) \tr_\omega \chi \equiv C^s.
    \end{align}
    Integration over $M$ determines the constant explicitly:
    \[
    C^s = s \underline{R} - (1-s) \underline{\chi},
    \]
    where the averages are given by:
    \[
    \underline{\chi} = \frac{m\int_M\chi\wedge\omega^{m-1}}{\int_M\omega^m}, \quad \underline{R} = \frac{\int_MR(\omega)\omega^m}{\int_M\omega^m}.
    \]
\end{defn}

\subsection{Twisted functionals and flows}

Define the space of K\"ahler potentials:
\[
\mathcal{H} = \{ \varphi \in C^\infty (M,\mathbb{R}) : \omega_\varphi := \omega + i\ddbar \varphi > 0 \}.
\]
Formally, the tangent space at $\varphi \in \mathcal{H}$ is $T_\varphi \mathcal{H} = C^\infty (M,\mathbb{R})$. The space $\mathcal{H}$ carries a natural Riemannian structure known as the \emph{Mabuchi metric}:
\[
\llangle f, g \rrangle_{\varphi} = \int_M f g \, \omega_\varphi^m, \quad \forall f, g \in C^\infty (M,\mathbb{R})= T_\varphi \mathcal{H}.
\]
The central energy functionals in our theory are defined as follows:
\begin{defn}[Twisted $K$-energy]
    For $s \in [0,1]$, the \emph{twisted $K$-energy functional} $\mM^s : \mathcal{H} \to \mathbb{R}$ is defined by:
    \[
    \mM^s(\varphi) = s \mM(\varphi) + (1-s) J_\chi(\varphi),
    \]
    where:
    \begin{itemize}
        \item $\mM : \mathcal{H} \to \mathbb{R}$ is the \emph{Mabuchi $K$-energy functional} \cite{MR0796483}, characterized by:
        \[
        D(\mM)_{\varphi}(\phi) = -\int_M \phi (R_\varphi - \underline{R}) \, \omega_\varphi^m;
        \]
        \item $J_\chi : \mathcal{H} \to \mathbb{R}$ is the \emph{$J$-functional} \cite{MR1772078}, characterized by:
        \[
        D(J_\chi)_{\varphi}(\phi) = \int_M \phi (\tr_{\omega_\varphi} \chi - \underline{\chi}) \, \omega_\varphi^m.\]
    \end{itemize}
\end{defn}
The twisted cscK equation \eqref{twisted cscK equaiton} is the Euler-Lagrange equation for the twisted $K$-energy functional $\mM^s$.

\begin{defn}[Twisted Calabi flow]
    The \emph{twisted Calabi flow} is the gradient flow of the twisted $K$-energy functional:
    \begin{align} \label{tcflow1}
        \frac{\partial \varphi}{\partial t} = s R_\varphi - (1-s) \tr_{\varphi} \chi - s \underline{R} + (1-s) \underline{\chi},
    \end{align}
    where $R_\varphi$ denotes the scalar curvature of $\omega + i\ddbar \varphi$, and $\tr_\varphi \chi = \tr_{\omega_\varphi} \chi$.
\end{defn}
When $s=0$, then twisted Calabi flow reduces to the $J$-flow introduced by Donaldson in \cite{donaldson1999moment},
\begin{align}
 \frac{\partial \varphi}{\partial t} =  -   \tr_{\varphi} \chi  +   \underline{\chi}.
\end{align}
We refer to \cite{MR2368374,MR3318155,MR3698234} for more details of $J$-flow. When $s=1$, then twisted Calabi flow reduces to usual Calabi flow
\begin{align}
 \frac{\partial \varphi}{\partial t} =   R_\varphi  -\underline{R}
\end{align}
We refer to \cite{MR2405167,MR4259154,MR1820328,MR3969453,MR3293741} for more studies. The twisted Calabi flow can be viewed as a continuous path connecting $J$-flow and Calabi flow.

Introducing the notation $\tR = s R_\varphi-(1-s) \tr_{\varphi} \chi $ and $\lR^s = -(1-s) \lchi + s \lR$, the twisted Calabi flow equation simplifies to:
\[
\frac{\partial \varphi}{\partial t} = R^s - \btR.
\]

A related but distinct flow is the twisted K\"ahler-Ricci flow:
\begin{defn}[Twisted K\"ahler-Ricci flow]
    The \emph{{twisted K\"ahler-Ricci flow}} is defined by the evolution equation:
    \[
    \frac{\partial g(t)}{\partial t} = \ric^s(\omega(t)).
    \]
\end{defn}
We refer to \cite{MR3518375} for more studies of this flow.

\subsection{The twisted Lichnerowicz operator}

The Lichnerowicz operator plays a fundamental role in the study of canonical metrics in K\"ahler geometry. We recall its classical definition before introducing the twisted generalization.

\begin{defn}[Classical Lichnerowicz operator]\label{Classical Lichnerowicz operator}
    The \emph{Lichnerowicz operator} is the fourth-order, self-adjoint, semi-positive differential operator defined by:
    \[
    \bL(f) = \mathfrak{D}^* \mathfrak{D} f = \Delta^2 f + \mathrm{Ric}^{\alpha\bar{\beta}} f_{\alpha\bar{\beta}} + g^{\alpha\bar{\beta}}   R_\alpha  f_{\bar{\beta}}, \quad \forall f \in C^\infty(M,\mathbb{C}),
    \]
    where $\mathfrak{D} f = \bar{\partial} \nabla^{1,0} f$.
\end{defn}

Originally introduced by Lichnerowicz to characterize holomorphic vector fields on K\"ahler manifolds, this operator admits extensions to symplectic settings \cite{he2023hermitian}. We refer to \cite[Section 1.23]{gauduchon2010calabi} for a comprehensive treatment.

We now introduce the central differential operator in our theory.

\begin{defn}[Twisted Lichnerowicz operator] \label{defoflich}
    For the fixed K\"ahler form $\chi$ and parameter $s \in [0,1]$, the \emph{twisted Lichnerowicz operator} $\bL^s$ is defined by:
    \[
    \bL^s(f) = s \bL(f) - (1-s) i \bar{\partial}^* (\nabla^{1,0} f \lrcorner \chi), \quad \forall f \in C^\infty(M,\mathbb{C}).
    \]
\end{defn}
Twisted Lichnerowicz operator is also a continuity path connecting the classical Lichnerowicz operator and the gradient of $J$-functional.
Using identity \eqref{equ:deltagrad}, we obtain an equivalent expression:
\begin{align}
		\label{lichexpression2}
	\begin{split}
	\bL^s(f) = &s  \bL(f) - (1-s) \langle i\ddbar f, \chi \rangle - (1-s) \langle \partial f, \bar{\partial}^* \chi \rangle
	\\
	=&s\Delta^2f+\langle \mathrm{Ric}^s, i\ddbar f\rangle+\langle \partial R^s, \partial f\rangle.
	\end{split}
\end{align}
Since $\chi$ is a positive $(1,1)$-form, it induces a Riemannian metric on the tangent bundle $TM$ via:
\[
\langle X, Y \rangle_\chi = \chi(X, JY), \quad \forall X, Y \in T_p M.
\]
This metric extends $\mathbb{C}$-linearly to $T^\mathbb{C} M$. In local coordinates, if $\chi = i\chi_{\alpha\bar{\beta}} dz_\alpha \wedge d\bar{z}_\beta$, then the induced Hermitian inner product is:
\[
\left\langle \frac{\partial}{\partial z_\alpha}, \frac{\partial}{\partial z_\beta} \right\rangle_\chi = \chi_{\alpha\bar{\beta}}.
\]
The associated global inner product is defined by:
\[
\llangle X, Y \rrangle_\chi = \int_M \langle X, Y \rangle_\chi \, \omega^m, \quad \forall X, Y \in \Gamma(T^\mathbb{C} M).
\]

The fundamental properties of the twisted Lichnerowicz operator are summarized in the following result.

\begin{lem} \label{lem:twistedlich}
    For any $\varphi, \phi \in C^\infty(M,\mathbb{R})$, we have:
    \[
    \llangle \bL^s(\varphi), \phi \rrangle = s \llangle \mathfrak{D} \varphi,  \mathfrak{D}  \phi \rrangle + (1-s) \llangle \nabla^{1,0} \varphi, \nabla^{1,0} \phi \rrangle_{\chi}.
    \]
    Consequently:
    \begin{enumerate}
        \item $\bL^s$ is a self-adjoint non-negative operator;
        \item If $s\in[0,1)$, then $\ker \bL^s = \{\text{constants}\}$;
        \item When restricted to $C^\infty_0(M,\mathbb{R})=\{f\in C^\infty(M,\bR), \int_Mf\omega^m=0\}$, we have $\ker \bL^s|_{C^\infty_0(M,\mathbb{R})} = \{0\}$.
    \end{enumerate}
\end{lem}

\begin{proof}
    The key computation involves the twisted term:
    \begin{align*}
        \int_M \langle i \bar{\partial}^* (\nabla^{1,0} \varphi \lrcorner \chi), \phi \rangle \, \omega^m
        &= \int_M i \langle \nabla^{1,0} \varphi \lrcorner \chi, \bar\partial \phi \rangle \, \omega^m   = \int_M -\varphi^{\alpha}   \chi_{\alpha\bar\beta}\phi^{\bar\beta}\, \omega^m  \\
        &= - \llangle \nabla^{1,0} \varphi, \nabla^{1,0} \phi \rrangle_\chi.
    \end{align*}
    Therefore,
    \begin{align*}
        \llangle \bL^s(\varphi), \phi \rrangle
        &= \llangle s \bL(f) - (1-s) i \bar{\partial}^* (\nabla^{1,0} \varphi \lrcorner \chi), \phi \rrangle \\
        &= s \llangle \mathfrak{D} \varphi,  \mathfrak{D}  \phi \rrangle + (1-s) \llangle \nabla^{1,0} \varphi, \nabla^{1,0} \phi \rrangle_\chi.
    \end{align*}
    The semi-positivity follows immediately, and $\llangle \bL^s(f), f \rrangle = 0$ if and only if $\nabla^{1,0} f = 0$, i.e., $f$ is constant.
\end{proof}

We conclude with an eigenvalue estimate that will be crucial for our stability analysis.

\begin{lem} \label{lem:1.6}
   Assume $s\in[0,1)$. Let $\lambda_1$ denote the first eigenvalue of $\bL^s$ and $\mu_1$ the first positive eigenvalue of the Laplacian $\Delta_\omega$. If
    \[
    \chi(X, JX) \geq \kappa \omega(X, JX), \quad \forall X \in \Gamma(TM),
    \]
    for some $\kappa > 0$, then:
    \begin{align} \label{1steigen}
        \lambda_1 \geq \kappa (1-s) \mu_1.
    \end{align}
\end{lem}

\begin{proof}
    The lower bound on $\chi$ implies:
    \[
    \| \mathrm{grad} f \|_{L^2(\chi)}^2 \geq \kappa \| \mathrm{grad} f \|_{L^2}^2.
    \]
    For any $f\in C^\infty (M,\bR)$, denote $c_\omega(f)=\int_Mf\omega^m/\int_M\omega^m$. Therefore,
    \begin{align*}
        \lambda_1 &= \inf_{\substack{f \in C^\infty(M,\mathbb{R}),\\ f\neq const}} \frac{\int_M \langle \bL^s(f-c_\omega(f)),f-c_\omega(f) \rangle \, \omega^m }{\int_M (f-c_\omega(f))^2 \, \omega^m } \\
        &= \inf_{\substack{f \in C^\infty(M,\mathbb{R}),\\ f\neq const}} \frac{ s \| \bar{\partial} \nabla^{1,0} f \|_{L^2}^2 + (1-s) \| \nabla^{1,0} f \|_{L^2(\chi)}^2 }{\| f-c_\omega(f) \|_{L^2}^2} \\
        &\geq \inf_{\substack{f \in C^\infty(M,\mathbb{R}),\\ f\neq const}}  \frac{ \kappa (1-s) \| \nabla^{1,0} f \|_{L^2}^2 }{\| f -c_\omega(f)\|_{L^2}^2} = \kappa (1-s) \mu_1.
    \end{align*}
\end{proof}

\section{Variational theory of the twisted Calabi functional}
\label{sec:variational-theory}
In his section, we compute the first and second variations of twisted Calabi functional, establish key convexity properties, and derive important geometric consequences for the twisted Calabi flow.

\subsection{The twisted Calabi functional and its first variation}
We consider the space of K\"ahler metrics within a fixed K\"ahler class.
The space can be identified with the space of K\"ahler potentials $\mathcal{H}/\{constants\}$, allowing us to view the twisted Calabi functional as:
\[
\mC^s(\varphi) = \frac{1}{2} \int_M (\tR(\omega_\varphi)-\underline{R}^s)^2 \, \omega_\varphi^m, \quad \varphi \in \mathcal{H}.
\]

Our first main result characterizes the critical points of this functional.

\begin{thm}[First variation] \label{thm:first-variation}
    Assume $s\in(0,1)$. The first variation of the twisted Calabi functional is given by:
    \[
    D(\mC^s)_\omega(\varphi) = -  \llangle\varphi, \bL^s(\tR)\rrangle, \quad \forall \varphi \in C^\infty(M,\mathbb{R}).
    \]
    Consequently, the metric $\omega$ is a critical point of $\mC^s$ if and only if it is a twisted cscK metric.
\end{thm}

\begin{proof}
    Consider a variation $\omega_t = \omega + t i\partial\bar{\partial}\varphi$ for $t \in (-\epsilon, \epsilon)$. Standard variational formulas in K\"ahler geometry (see \cite[Section 2.1]{MR1787650}, \cite[Theorem 4.2]{MR3186384}, \cite{gauduchon2010calabi}) yield:
    \begin{align*}
        \frac{\partial}{\partial t}\Big|_{t=0} \tr_{\omega_t} \chi &= -\langle i\ddbar\varphi, \chi\rangle,\quad
       \frac{\partial}{\partial t}\Big|_{t=0} \omega_t^m  = \Delta \varphi \, \omega^m, \\
        \frac{\partial}{\partial t}\Big|_{t=0} R(\omega_t) &= -\Delta^2 \varphi - \langle \ric, i\ddbar \varphi \rangle = -\mD^*\mD \varphi + \langle \partial R,  {\partial} \varphi \rangle.
    \end{align*}
    Combining these results, we obtain:
    \begin{align}\label{varoftwistedscalar}
    \frac{d}{dt}\Big|_{t=0} R^s = -s \bL(\varphi) + s \langle \partial R , \partial\varphi \rangle + (1-s) \langle i\partial\bar{\partial} \varphi, \chi \rangle.
    \end{align}
    Now computing the variation of the twisted Calabi functional (omitting the volume form $\omega^m$ for brevity), we have
    \begin{align*}
        &\frac{d}{dt}\Big|_{t=0} \mC^s(\omega + t i\ddbar \varphi) \\
        &= \int_M \left( -s \mD^*\mD \varphi +  s \langle \partial R , \partial\varphi \rangle + (1-s) \langle i\partial\bar{\partial} \varphi, \chi \rangle \right) (\tR-\underline{R}^s)\\
        &
         + \frac{1}{2} \int_M (\tR-\underline{R}^s)^2 \Delta \varphi \\
        &= \int_M \left( -s \bL(\varphi) +   s \langle \partial R , \partial\varphi \rangle + (1-s) \langle i\partial\bar{\partial} \varphi, \chi \rangle -   \langle \partial R^s , \partial\varphi \rangle\right) (\tR-\underline{R}^s).
        \end{align*}
        According to the expression of twisted Lichnerowicz operator \eqref{lichexpression2}, we have
        \begin{align*}
         &-s \bL(\varphi) +   s \langle \partial R , \partial\varphi \rangle + (1-s) \langle i\partial\bar{\partial} \varphi, \chi \rangle -   \langle \partial R^s , \partial\varphi \rangle
         \\
         =&-s \bL(\varphi)+(1-s)\left( \langle i\ddbar \varphi, \chi \rangle + \langle  \partial \tr_\omega \chi ,  {\partial} \varphi \rangle \right)\\
         =&-\bL^s(\varphi).
        \end{align*}
        Hence,
        \begin{align*}
        D(\mC^s)_\omega(\varphi)
        &= -\int_M \bL^s(\varphi) (\tR-\underline{R}^s) = -\llangle\varphi, \bL^s(\tR)\rrangle.
    \end{align*}
    Therefore, $\omega$ is a critical point if and only if $\bL^s(R^s) = 0$, which occurs precisely when $\tR$ is constant, i.e., when $\omega$ is a twisted cscK metric.
\end{proof}

\begin{rem}
    In the classical case $s = 1$, the critical points of the Calabi functional are extremal K\"ahler metrics, as established by Calabi \cite{MR645743}. However, for $s \neq 1$, the critical points are exactly the twisted cscK metrics. This reveals a fundamental structural difference: twisted extremal K\"ahler metrics do not arise as critical points of the twisted Calabi functional.
\end{rem}

The first variation formula immediately yields monotonicity of twisted Calabi functional along the twisted Calabi flow.

\begin{cor}[Energy decreases along the flow] \label{cor:energy-decrease}
    The twisted Calabi energy is strictly decreasing along the twisted Calabi flow, except at critical points.
\end{cor}

\begin{proof}
    Along the twisted Calabi flow $\frac{\partial \varphi}{\partial t} = \tR(\varphi) - \btR$, we have:
    \begin{align*}
        \frac{d}{dt} \mC^s(\varphi(t))
        &= -\int_M (\tR(\varphi) - \btR) \bL^s(\tR(\varphi)) \, \omega_\varphi^m \\
        &= -\int_M (\tR(\varphi) - \btR)\bL^s(\tR(\varphi) - \btR) \, \omega_\varphi^m \leq 0,
    \end{align*}
    where the inequality follows from the semi-positivity of $\bL^s$ (\cref{lem:twistedlich}). Equality holds if and only if $\bL^s(\tR(\varphi)) = 0$, i.e., at twisted cscK metrics.
\end{proof}

\subsection{Geometric properties of the twisted Calabi flow}

The twisted Calabi flow exhibits favorable geometric properties with respect to the Mabuchi metric geometry on $\mathcal{H}$.

\begin{thm}[Distance Decreasing Property] \label{thm:distance-decreasing}
    Assume $s\in(0,1)$. Let $\varphi(\tau): [0,1] \to \mathcal{H}$ be a smooth curve, and let $\varphi(\tau,t)$ denote its deformation under the twisted Calabi flow at time $t$. Denote by $l(t)$ the length of the curve $\varphi_t(\tau) := \varphi(\tau,t)$. Then:
    \[
    \frac{d l(t)}{d t} = -\int_0^1 \left( \int_M \left( \frac{\partial \varphi}{\partial \tau} \right)^2 \omega_\varphi^m \right)^{-\frac{1}{2}} \left( \int_M \frac{\partial \varphi}{\partial \tau} \bL^s \left( \frac{\partial \varphi}{\partial \tau} \right) \omega_\varphi^m \right) d\tau \leq 0.
    \]
    Consequently, the twisted Calabi flow strictly decreases distances in $\mathcal{H}$, unless the curve degenerates to a point.
\end{thm}

\begin{proof}
    For each $\tau \in [0,1]$, the function $\varphi(\tau,t)$ satisfies the twisted Calabi flow equation:
    \begin{align*}
        \begin{cases}
            \frac{\partial \varphi(\tau,t)}{\partial t} = R^s(\varphi(\tau,t)) - \underline{R}^s, \\
            \varphi(\tau,0) = \varphi(\tau).
        \end{cases}
    \end{align*}
    The energy of the curve $\varphi_t(\tau) = \varphi(\tau,t)$ is given by:
    \[
    E(t) = \int_0^1 \int_M \left( \frac{\partial \varphi}{\partial \tau} \right)^2 \omega_{\varphi(\tau,t)}^m d\tau.
    \]
    Differentiating with respect to $t$:
    \begin{align*}
        \frac{d E(t)}{d t}
        &= \int_0^1 \int_M 2 \frac{\partial \varphi}{\partial \tau} \frac{\partial^2 \varphi}{\partial \tau \partial t} \omega_\varphi^m d\tau + \int_0^1 \int_M \left( \frac{\partial \varphi}{\partial \tau} \right)^2 \Delta_\varphi \left( \frac{\partial \varphi}{\partial t} \right) \omega_\varphi^m d\tau \\
        &= \int_0^1 \int_M 2 \frac{\partial \varphi}{\partial \tau} \frac{\partial}{\partial \tau} (R^s(\varphi) - \underline{R}^s) \omega_\varphi^m d\tau \\
        &\quad + \int_0^1 \int_M \left( \frac{\partial \varphi}{\partial \tau} \right)^2 \Delta_\varphi (R^s(\varphi) - \underline{R}^s) \omega_\varphi^m d\tau.
    \end{align*}
    Using the variation formula \eqref{varoftwistedscalar}:
    \[
    \frac{\partial}{\partial \tau} (R^s(\varphi) - \underline{R}^s) = -s \bL \left( \frac{\partial \varphi}{\partial \tau} \right) + s \langle \partial R, \bar{\partial} \frac{\partial \varphi}{\partial \tau} \rangle + (1-s) \langle i\ddbar \frac{\partial \varphi}{\partial \tau}, \chi \rangle,
    \]
    we obtain:
    \begin{align*}
        \frac{d E(t)}{d t}
        &= \int_0^1 \int_M 2 \frac{\partial \varphi}{\partial \tau} \left( -s \bL \frac{\partial \varphi}{\partial \tau} + s \langle \partial R, \bar{\partial} \frac{\partial \varphi}{\partial \tau} \rangle + (1-s) \langle i\ddbar \frac{\partial \varphi}{\partial \tau}, \chi \rangle \right) \omega_\varphi^m d\tau \\
        &\quad - \int_0^1 \int_M 2 \frac{\partial \varphi}{\partial \tau} \langle \bar{\partial} \frac{\partial \varphi}{\partial \tau}, \partial R^s \rangle \omega_\varphi^m d\tau \\
        &= \int_0^1 \int_M 2 \frac{\partial \varphi}{\partial \tau} \left( -s \bL \frac{\partial \varphi}{\partial \tau} + (1-s) \left( \langle \bar{\partial} \frac{\partial \varphi}{\partial \tau}, \partial \tr_\varphi \chi \rangle + \langle i\ddbar \frac{\partial \varphi}{\partial \tau}, \chi \rangle \right) \right) \omega_\varphi^m d\tau \\
        &= -2 \int_0^1 \int_M \frac{\partial \varphi}{\partial \tau} \bL^s \left( \frac{\partial \varphi}{\partial \tau} \right) \omega_\varphi^m d\tau \leq 0.
    \end{align*}
    Let $l(s)$ be the length of $\varphi(t,s)$, i.e.,
    \begin{align*}
    	l(t) = \int_0^1\sqrt{\int_M \left(\frac{\partial\varphi}{\partial \tau} \right)^2\omega_{\varphi}^m}\ d\tau.
    \end{align*}
    Then we have
    \begin{align*}
    	\frac{dl(t)}{dt}=&-\int_0^1\left(\int_M \left(\frac{\partial\varphi}{\partial \tau} \right)^2\omega_{\varphi}^m\right)^{-\frac{1}{2}}\left(\int_M \frac{\partial\varphi}{\partial \tau}\bL^s\left(\frac{\partial\varphi}{\partial \tau}\right)\omega_{\varphi }^m\right)d\tau\leq 0.
    \end{align*}
    From this formula,
    if the length of a smooth curve is not decreasing, then
    \begin{align*}
    	\int_M \frac{\partial\varphi}{\partial \tau}\bL^s\left(\frac{\partial\varphi}{\partial \tau}\right)\omega_{\varphi }^m
    	= s\|\mD\frac{\partial\varphi}{\partial \tau}\|_{L^2}^2+(1-s)\|\partial \frac{\partial\varphi}{\partial \tau}\|_{L^2(\chi)}^2\equiv 0,
    \end{align*}
    i.e.,
    \begin{align*}
    	\frac{\partial\varphi}{\partial \tau}\equiv 0,
    \end{align*}
    equivalently, the curve $\varphi(s)$  degenerates to a point.

%
\end{proof}

\begin{rem}
    In the classical setting, Calabi and Chen \cite{MR4521129} proved that the Calabi flow decreases the distance of any two points in $\mH$, with the exception of two cases: when $\varphi_\tau'$ is a holomorphic potential, or when the curve degenerates. In the twisted setting, the strict positivity of $\bL^s$ eliminates the first exceptional case, resulting in stronger distance-decreasing properties.
\end{rem}

\subsection{Second variation and convexity analysis}

We now analyze the second variation of the twisted Calabi functional at critical points, revealing its strong convexity properties.

\begin{thm}[Second variation] \label{thm:hessian-formula}
     Assume $s\in(0,1)$. At a critical point of the twisted Calabi functional, the Hessian is given by:
    \[
    \mathrm{Hess} \, \mC^s(\varphi, \phi) = \int_M \langle \bL^s(\varphi),  {\bL}^s(\phi) \rangle \,\omega^m , \quad \forall \varphi, \phi \in C^\infty(M,\mathbb{R}).
    \]
    In particular, $\mathrm{Hess} \, \mC^s$ is strictly positive definite on $\{f:f\in C^\infty_0(M,\bR):\int_Mf\omega^m=0\}$ at critical points.
\end{thm}
\begin{proof}
    Consider the two-parameter family $\omega(\sigma, \tau) = \omega + \sigma i\ddbar \varphi + \tau i\ddbar \phi$, and denote $\tR(\sigma,\tau) = \tR(\omega(\sigma,\tau))$. The second variation is:
    \begin{align*}
        &D^2 \mC^s(\varphi, \phi) = \frac{\partial^2}{\partial \sigma \partial \tau} \Big|_{\sigma,\tau=0} \mC^s(\omega + \sigma i\ddbar \varphi + \tau i\ddbar \phi) \\
        &= \frac{\partial}{\partial \tau} \Big|_{\tau=0} \int_M -s \langle \mD \varphi,  \mD \tR \rangle (\omega + \tau i\ddbar \phi)^m \\
        &\quad + (1-s) i \frac{\partial}{\partial \tau} \Big|_{\tau=0} \int_M \chi(\nabla^{1,0} \varphi, \nabla^{1,0} \tR)(\omega + \tau i\ddbar \phi)^m .
    \end{align*}
    The key step involves computing the variation of $\nabla^{0,1} \tR$. Consider  the identity
    \begin{align}\label{keyidentity}
     \nabla^{1,0} \tR \lrcorner \omega =  i \bar\partial \tR.
    \end{align}
 The variation of the left hand side of \eqref{keyidentity} is
    \begin{align*}
    \frac{\partial}{\partial \tau} \Big|_{\tau=0} (\nabla^{1,0} \tR \lrcorner \omega(0,\tau))=i \bar\partial \frac{\partial}{\partial \tau} \Big|_{\tau=0}\tR(0,\tau)=\nabla^{1,0}\left(\frac{\partial}{\partial \tau} \Big|_{\tau=0}\tR(0,\tau)\right)  \lrcorner \omega(0,0).
    \end{align*}
    Since $\tR(0,0)=const$, the variation of the right hand side of \eqref{keyidentity} is given by
    \begin{align*}
        \frac{\partial}{\partial \tau} \Big|_{\tau=0} (\nabla^{1,0} \tR \lrcorner \omega )
        &= \frac{\partial}{\partial \tau} \Big|_{\tau=0} (\nabla^{1,0} \tR) \lrcorner \omega + \nabla^{1,0} \tR \lrcorner i\ddbar \phi  = \frac{\partial}{\partial \tau} \Big|_{\tau=0} (\nabla^{1,0} \tR) \lrcorner \omega .
    \end{align*}
    Hence
    \begin{align*}
    \frac{\partial}{\partial \tau} \Big|_{\tau=0} (\nabla^{1,0} \tR)=\nabla^{1,0}  \frac{\partial}{\partial \tau} \Big|_{\tau=0}\tR.
    \end{align*}
    Combining this with the first variation formula \eqref{varoftwistedscalar},
    we obtain:
    \begin{align*}
        \frac{\partial}{\partial \tau} \Big|_{\tau=0} (\nabla^{1,0} \tR)
        &= \nabla^{1,0} \left( -s \bL(\phi) + s \langle \partial R, {\partial} \phi \rangle + (1-s) \langle i\ddbar \phi, \chi \rangle  \right).
    \end{align*}

    Now we compute the Hessian:
    \begin{align*}
        \mathrm{Hess} &\, \mC^s(\varphi, \phi) \\
        &= \int_M -s \langle \mD \varphi,  {\mD} \left( -s \bL(\phi) + s \langle \partial R , \partial\phi \rangle + (1-s) \langle i\partial\bar{\partial} \phi, \chi \rangle \right) \rangle \\
        &\quad + (1-s) \int_M \chi(\nabla^{1,0} \varphi, \nabla^{0,1} \left( -s \bL(\phi) + s \langle \partial R , \partial\phi \rangle + (1-s) \langle i\partial\bar{\partial} \phi, \chi \rangle\right)) \\
        &= -\int_M \langle \bL^s(\varphi), -s \bL(\phi) + s \langle \partial R , \partial\phi \rangle + (1-s) \langle i\partial\bar{\partial} \phi, \chi \rangle \rangle.
    \end{align*}
    Note that
$\tR=sR-(1-s)\mathrm{tr}_\omega\chi=const$ implies that
$$
s\langle \partial R , \partial\phi \rangle= (1-s)\langle \partial\mathrm{tr}_\omega\chi , \partial\phi \rangle.
$$
Combining \cref{lem:tensoridentity}, we have
\begin{align*}
-s \bL(\phi) + s \langle \partial R , \partial\phi \rangle + (1-s) \langle i\partial\bar{\partial} \phi, \chi \rangle=-\bL^s(\phi).
\end{align*}
Hence
 \begin{align*}
	\mathrm{Hess} &\, \mC^s(\varphi, \phi)
 = \int_M  \langle \bL^s(\phi),  \bL^s(\varphi) \rangle\, \omega^m.
\end{align*}
    The strict positivity follows from the injectivity of $\bL^s$ on $C^\infty_0(M,\mathbb{R})$.
\end{proof}

\begin{rem}
	For the Calabi functional, its Hessian at critical point(see \cite{MR780039})  is given by
	\begin{align*}
		\mathrm{Hess}\, \mC(\varphi,\phi)=\llangle \overline{\bL}\bL\varphi,\phi\rrangle=\llangle \bL\varphi,\overline{\bL}\phi\rrangle,
	\end{align*}
	and the two operators $\bL, \overline{\bL}$ commute if $\omega$ is an extremal K\"ahler metric. The conjugate of the twisted Lichnerowicz operator does not appear in the expression of the $\mathrm{Hess}\,\mC^s$. This is because, at a twisted cscK metric, the twisted Lichnerowicz operator  $$
	\bL^s(f)=s\Delta^2f+\langle \mathrm{Ric}^s, i\ddbar f\rangle+\langle \partial R^s, \partial f\rangle=s\Delta^2f+\langle \mathrm{Ric}^s, i\ddbar f\rangle
	$$
	is a real operator, i.e., $\bL^s=\overline{\bL^s}$.
\end{rem}

\subsection{Geometric consequences and the twisted  $K$-energy}

The convexity results have important implications for the geometry of the space of K\"ahler metrics and the behavior of the twisted  $K$-energy.

\begin{cor}[Isolation of critical points] \label{cor:isolation}
    Assume $s\in(0,1)$. Twisted cscK metrics, when they exist, are isolated within their K\"ahler class.
\end{cor}

\begin{proof}
    The strict positive definiteness of $\mathrm{Hess} \, \mC^s$ at critical points prevents the existence of non-trivial degenerations in a neighborhood of any twisted cscK metric.
\end{proof}

\begin{remark}
\cref{cor:isolation} is a straightforward consequence of the convexity of the twisted Calabi functional. By approximation of weak geodesics, Berman-Berndtsson \cite{MR3671939}(see also Chen-Li-P\u{a}uni \cite{MR3582114}) proved that there exists at most one twisted cscK metric in each K\"ahler class. The convexity properties in the twisted setting characterize the local behavior of twisted cscK metrics.
\end{remark}
This contrasts sharply with the classical case, where the connected components of extremal K\"ahler metrics form orbits of $\mathrm{Aut}^\Omega_0$, the identity component of the automorphism group preserving the K\"ahler class.

The convexity properties extend to the twisted  $K$-energy along geodesics in $\mathcal{H}$.

\begin{cor}[Convexity of twisted $K$-energy] \label{thm:convexity-K-energy}
     Assume $s\in(0,1)$. The second derivative of the twisted  $K$-energy along a smooth geodesic $\varphi(t)$ in $\mathcal{H}$ satisfies:
    \[
    \frac{d^2 \mM^s(t)}{d t^2} = \int_M \varphi'(t) \bL^s(\varphi'(t)) \, \omega_{\varphi(t)}^m \geq 0.
    \]
    Consequently:
    \begin{enumerate}
        \item The twisted  $K$-energy is convex along smooth geodesics in $\mathcal{H}$;
        \item Distinct twisted cscK metrics cannot be connected by a smooth geodesic.
    \end{enumerate}
\end{cor}

\begin{proof}
    Let $\varphi(t)$ be a geodesic in $\mathcal{H}$, satisfying the geodesic equation:
    \[
    \varphi''(t) - \langle \partial \varphi'(t), \partial \varphi'(t) \rangle_{\varphi(t)} = 0.
    \]
    Differentiating the twisted K-energy along the geodesic:
    \begin{align*}
        \frac{d^2 \mM^s(t)}{d t^2}
        &= \frac{d}{d t} \int_M -\varphi'(t) (R^s - \underline{R}^s) \, \omega_{\varphi(t)}^m \\
        &= \int_M -\varphi''(t) (R^s - \underline{R}^s) + \varphi'(t) (R^s - \underline{R}^s) \Delta \varphi'(t) \\
        &\quad - \varphi'(t) \left( -s \bL(\varphi') + s \langle \partial R, \bar{\partial} \varphi' \rangle + (1-s) \langle i\ddbar \varphi', \chi \rangle \right) \omega_{\varphi(t)}^m.
    \end{align*}
    Using the geodesic equation and rearranging terms:
    \begin{align*}
        \frac{d^2 \mM^s(t)}{d t^2}
        &= \int_M \left( -\varphi''(t) + \langle \partial \varphi'(t), \bar{\partial} \varphi'(t) \rangle_{\varphi(t)} \right) (R^s - \underline{R}^s) \\
        &\quad - \varphi'(t) \left( -s \bL \varphi' + (1-s) \left( \langle \bar{\partial} \varphi'(t), \partial \tr_\varphi \chi \rangle + \langle i\ddbar \varphi'(t), \chi \rangle \right) \right) \omega_{\varphi(t)}^m \\
        &= \int_M \varphi'(t) \bL^s(\varphi'(t)) \, \omega_{\varphi(t)}^m \geq 0.
    \end{align*}

    For the second statement, suppose $\varphi(t)$ is a smooth geodesic connecting two distinct twisted cscK metrics $\omega_{\varphi(0)}$ and $\omega_{\varphi(1)}$. Let $f(t) = \frac{d \mM^s(\varphi(t))}{d t}$. Then $f(0) = f(1) = 0$ and $f'(t) \geq 0$ for all $t \in [0,1]$, implying $f(t) \equiv 0$. Consequently, $\bL^s(\varphi'(t)) \equiv 0$, so $\varphi'(t) \equiv 0$, contradicting the assumption that the endpoints are distinct.
\end{proof}

\begin{rem}
	The convexity of the K-energy along smooth geodesics was first established by Mabuchi \cite{MR909015}, while Chen proved the convexity of the
	$J$-functional along weak geodesics \cite{MR1863016}.
	Subsequent works \cite{MR3671939, MR3582114} extended this convexity to the K-energy along weak geodesics.
	The convexity of the twisted K-energy in the metric completion of the space of K\"ahler potentials $\mH$ with respect to the
$L^p$-type path length metric  $d_p$ was further demonstrated in \cite{MR3687111}. By leveraging the twisted Lichnerowicz operator, we provide a concise and explicit formulation of the convexity of the twisted K-energy along smooth geodesics.
\end{rem}

\section{Short-time existence of the twisted Calabi flow}
\label{sec:short-time-existence}

This section establishes the local well-posedness of the twisted Calabi flow on compact K\"ahler manifolds. We employ analytic semigroup methods in little H\"older spaces, adapting the framework developed for the classical Calabi flow to accommodate the additional twisted terms. Our proof mainly follows He's method in \cite{MR3010280}.

\subsection{Main existence theorem}
We begin by introducing the function spaces essential for our analysis. Let $J = [0,T]$ and $\tilde{J} = (0,T]$ for some $T > 0$. For a Banach space $E$, we define the weighted spaces:

\begin{definition}[Weighted function spaces]
	For a Banach space $E$, we define:
	\begin{align*}
		C_{1/2}(J, E) &= \left\{u \in C(\tilde{J}, E) \mid t \mapsto t^{1/2}u(t) \in C(J,E),\ \lim_{t \to 0} t^{1/2}|u(t)|_E = 0 \right\}, \\
		C^1_{1/2}(J, E) &= \left\{u \in C^1(\tilde{J}, E) \mid u, \dot{u} \in C_{1/2}(J, E) \right\}.
	\end{align*}
	The corresponding norms are given by:
	\begin{align*}
		|u|_{C_{1/2}(J, E)} &= \sup_{t \in J} |t^{1/2}u(t)|_E, \\
		|u|_{C^1_{1/2}(J, E)} &= \sup_{t \in \tilde{J}} t^{1/2}(|\dot{u}(t)|_E + |u(t)|_E).
	\end{align*}
\end{definition}
Let $c^{k,\alpha}(M)$ be  the complement of $C^\infty(M)$ in $C^{k,\alpha}(M)$ under the usual norm $\|.\|_{C^{k,\alpha}}$. Thus $c^{k,\alpha}(M)$ is the Banach sub-space of $C^{k,\alpha}(M)$ with same norm.
We work with the following specific Banach spaces:
\begin{align*}
	E_0 &= c^{k-2,\alpha}(M), \quad E_1 = c^{k+2,\alpha}(M), \quad E_{1/2} = c^{k,\alpha}(M).
\end{align*}
$E_{1/2} = c^{k,\alpha}(M)$ is the interpolation space $(E_0,E_1)_{1/2}$.

Define the solution spaces:
\begin{align*}
	E_0(J) &= C_{1/2}(J, E_0), \\
	E_1(J) &= C^1_{1/2}(J, E_0) \cap C_{1/2}(J, E_1),
\end{align*}
with the norm:
\[
|u|_{E_1(J)} := \sup_{t \in \tilde{J}} t^{1/2}(|\dot{u}(t)|_{E_0} + |u(t)|_{E_1}).
\]

\begin{theorem}[Short-time existence]\label{thm:short-time-existence}
	Assume $s\in(0,1]$.
	For any smooth K\"ahler metrics $\omega, \chi$ on $M$ and parameter $s \in (0,1]$, there exist positive constants $T = C(\omega,\chi )s^2 > 0$ and $\epsilon = \epsilon(\omega,\chi) > 0$ such that for any initial data $x \in c^{k,\alpha}(M)$ for any $k\geq 2$ with $\|x\|_{C^{k,\alpha}} \leq \epsilon$, the twisted Calabi flow equation
	\[
	\begin{cases}
		\partial_t \varphi = R^s(\omega_\varphi) - \underline{R}^s, \\
		\varphi(0) = x
	\end{cases}
	\]
	admits a unique solution $\varphi(t, x) \in E_1([0,T])$. Moreover, $\varphi \in C([0,T], c^{k,\alpha}(M))\cap C((0,T), C^\infty(M))$ and for $x, y \in B_{c^{k,\alpha}(M)}(0,\epsilon)$, the corresponding solutions $\varphi(t,x), \varphi(t,y)$ satisfy the stability estimates:
	\begin{align}
		\|\varphi(t,x) - \varphi(t,y)\|_{C([0,T], c^{k,\alpha}(M))} &\leq c\|x - y\|_{c^{k,\alpha}(M)}, \label{eq:stability1} \\
		\|\varphi(t,x) - \varphi(t,y)\|_{E_1([0,T])} &\leq c\|x - y\|_{c^{k,\alpha}(M)}. \label{eq:stability2}
	\end{align}
\end{theorem}
\begin{rem}
Taking $k=2$ in \cref{thm:short-time-existence} implies \cref{t1.7}. Here we prove a more general result for initial value in $c^{k,\alpha}(M)$ with different $k$. \eqref{eq:stability2} indicates that the higher the regularity of the initial value, the higher the regularity of the solution.
The estimate \eqref{eq:stability2} plays an important role in the proof of \cref{thm:stability}.
\end{rem}
\subsection{Proof strategy and technical setup}

The proof employs a fixed point argument in carefully chosen function spaces. We reformulate the twisted Calabi flow as a quasi-linear parabolic equation.

Let
\[
F(\varphi) = s\Delta^2\varphi + (R^s - \underline{R}^s) = s\Delta^2\varphi + sR_\varphi - (1-s)\mathrm{tr}_\varphi\chi - \underline{R}^s,
\]  where $\Delta$ denotes the Laplacian operator with respect to the metric $\omega$.
We consider the equivalent equation:
\[
\begin{cases}
	\partial_t \phi + s\Delta^2 \phi = F(\varphi), \\
	u(0) = x.
\end{cases}
\]
Via the rescaling $\tau = ts$, $v(\tau) = \phi(t)$, we obtain the normalized equation:
\begin{equation}
	\label{eq:rescaled-flow}
	\begin{cases}
		\partial_\tau v + \Delta^2 v = \frac{1}{s}F(\varphi), \\
		v(0) = x.
	\end{cases}
\end{equation}
Assume $\epsilon>0$ is small enough. For any $x\in B_{E_{1/2}}(0,\epsilon)$, we define the admissible set:
\begin{align*}
	V_x(J) = \{&v \in E_1(J) : v(0) = x,\ |v|_{C(J,E_{1/2})} \leq \epsilon_0\} \cap B_{E_1(J)}(0, \epsilon_0),
\end{align*}
where $\epsilon_0 = c_0\epsilon$ for a constant $c_0$ depending only on $\omega$. Explicitly, any $v \in V_x(J)$ satisfies $v(0) = x$ and:
\begin{align}
	\sup_{t \in [0,T]} |v(t)|_{c^{k,\alpha}(M)} &\leq \epsilon_0, \label{eq:cond1} \\
	\sup_{t \in (0,T]} t^{1/2}(|v(t)|_{c^{k+2,\alpha}(M)} + |v'(t)|_{c^{k-2,\alpha}(M)}) &\leq \epsilon_0. \label{eq:cond2}
\end{align}

\begin{lemma}[Non-emptiness of admissible set]\label{lem:nonempty}
	For sufficiently small $T > 0$, the set $V_x(J)$ is non-empty.
\end{lemma}

\begin{proof}
	Consider the equation for some $x\in E_0$
	\begin{align}
		\label{equ2.1}
		\begin{cases}
			\frac{\partial u}{\partial t} = -\Delta^2u;\\
			u(0)=x
		\end{cases}
	\end{align}
	Then $\Delta^2:E_1\to E_0$ induces an analytic semigroup $e^{-t\Delta^2}:E_0\to E_1$. The solution of \eqref{equ2.1} is given by $u(t) = e^{-t\Delta^2}x$ for $t\in J$.
	According to \cite[Lemma 2.1, Lemma 2.2]{MR1838320}, we have the following three facts.
	\begin{enumerate}
		\item For $x\in E_{1/2}$, we have the equivalent norm
		\begin{align*}
			|x|_{E_{1/2}}:=\sup_{t\in(0,T]}t^{1/2}|\Delta^2(e^{-t\Delta^2}x)|_{E_0 };
		\end{align*}
		\item $[t\to e^{-t\Delta^2}x]\in E_1(J)$ and there exists a constant $c_1>0$ independent of $J$ such that for any $t\in J$,
		\begin{align}\label{equa:3.3}
			|e^{-t\Delta^2}x|_{E_1(J)}\leq c_1|x|_{E_{1/2}};
		\end{align}

		\item If $u\in E_1(J)$ with $u(0)=0$, then there exists a constant $c_2$ independent of $J$ such that
		\begin{align}\label{eq:3.4}
			|u|_{C(J,E_{1/2})}\leq c_2|u|_{E_1(J)}.
		\end{align}
	\end{enumerate}
	The constant $c_1, c_2$ only depend on $ \omega $. Also if $T$ is small, by the strong continuity of the semigroup $\{e^{-t\Delta^2}, t\geq 0\}$, we can get that
	\begin{align}\label{equa:3.4}
		|e^{-t\Delta^2}x-x|_{E_{1/2}}\leq \frac{1}{4}\epsilon_0.
	\end{align}
	Let $c_0=2c_1+2$ and $\epsilon_0=c_0\epsilon$. According to \eqref{equa:3.3}, for any $x\in B_{E_{1/2}}(0,\epsilon)$ we have
	\begin{align*}
		|e^{-t\Delta^2}x |_{E_{1}(J)}\leq c_1\epsilon<\frac{\epsilon_0}{2}.
	\end{align*}
	Also by \eqref{equa:3.4} we obtain
	\begin{align*}
		|e^{-t\Delta^2}x |_{E_{1/2}}\leq |x|_{E_{1/2}}+|e^{-t\Delta^2}x-x|_{E_{1/2}}<\frac{\epsilon_0}{2}+\frac{\epsilon_0}{4}<\epsilon_0
	\end{align*}
	provided $T$ is small. It follows that $V_x(J)$ is not empty.
\end{proof}

\subsection{Fixed point argument}

Define the solution map $\Pi: V_x(J) \to V_x(J)$ by $\Pi(\varphi) = v$, where $v$ solves:
\begin{align}\label{tcfdef}
	\begin{cases}
		\partial_\tau v + \Delta^2 v = \frac{1}{s}F(\varphi), \\
		v(0) = x.
	\end{cases}
\end{align}
The solution is given by:
\[
v(t) = e^{-t\Delta^2}x + \frac{1}{s}K(F)(t), \quad \text{where } K(F)(t) = \int_0^t e^{-(t-\tau)\Delta^2}F(\varphi(\tau))d\tau.
\]

\begin{lemma}[Well-definedness of solution map]\label{lem:solution-map}
	For sufficiently small $T > 0$ and $\epsilon > 0$, the map $\Pi: V_x(J) \to V_x(J)$ defined by $\Pi(\varphi) = v$, where $v$ solves
	\[
	\begin{cases}
		\partial_\tau v + \Delta^2 v = \frac{1}{s}F(\varphi), \\
		v(0) = x,
	\end{cases}
	\]
	is well-defined. That is, for any $\varphi \in V_x(J)$, we have $\Pi(\varphi) \in V_x(J)$.
\end{lemma}

\begin{proof}
	We need to verify that for $\varphi \in V_x(J)$, the solution $v = \Pi(\varphi)$ satisfies the conditions \eqref{eq:cond1} and \eqref{eq:cond2} defining $V_x(J)$.

	\emph{Step 1: Expression for the solution.}

	The solution is given by:
	\[
	v(t) = e^{-t\Delta^2}x + \frac{1}{s}K(F)(t), \quad \text{where } K(F)(t) = \int_0^t e^{-(t-\tau)\Delta^2}F(\varphi(\tau))d\tau.
	\]
	The analytic semi-group theorem implies that $v\in C((0,T), C^\infty(M))$ for any $T>0$.
	Recall that:
	\[
	F(\varphi) = s\Delta^2\varphi + R^s(\omega_\varphi) - \underline{R}^s.
	\]
	We analyze the twisted scalar curvature term in local coordinates. The scalar curvature $R_\varphi$ of $\omega_\varphi$ satisfies:
	\[
	R_\varphi = -g^{i\bar{j}}_\varphi \partial_i\partial_{\bar{j}}\log\det(g_{k\bar{l}} + \varphi_{k\bar{l}}).
	\]
	Expanding this expression:
	\begin{align*}
		R_\varphi = &-g^{i\bar{j}}_\varphi g^{k\bar{l}}_\varphi (\partial_i\partial_{\bar{j}}g_{k\bar{l}} + \varphi_{i\bar{j}k\bar{l}}) \\
		&+ g^{i\bar{j}}_\varphi g^{k\bar{q}}_\varphi g^{p\bar{l}}_\varphi (\partial_i g_{p\bar{q}} + \varphi_{ip\bar{q}})(\partial_{\bar{j}}g_{k\bar{l}} + \varphi_{\bar{j}k\bar{l}}).
	\end{align*}

	Similarly, the trace term expands as:
	\[
	\mathrm{tr}_\varphi\chi = g^{i\bar{j}}_\varphi \chi_{i\bar{j}}.
	\]
	The bi-Lapalcian expands as
	\begin{align*}
		\Delta^2\varphi=&g^{i\bar j}\partial_i\partial_{\bar j}(g^{k\bar l}\varphi_{k\bar l})\\
		=&g^{i\bar j} g^{k\bar l}\varphi_{i\bar jk\bar l} +g^{i\bar j}\partial_i  (g^{k\bar l})\varphi_{\bar jk\bar l}
		+g^{i\bar j}\partial_{\bar j} (g^{k\bar l})\varphi_{ik\bar l}+
		g^{i\bar j}\partial_i\partial_{\bar j}(g^{k\bar l})\varphi_{k\bar l}.
	\end{align*}
	Therefore,
	\begin{align}\label{express:F}
		\begin{split}
			F(\varphi) = &s\left[g^{i\bar{j}}g^{k\bar{l}} - g^{i\bar{j}}_\varphi g^{k\bar{l}}_\varphi\right]\varphi_{i\bar{j}k\bar{l}}
			+s g^{i\bar j}\left(\partial_i  (g^{k\bar l})\varphi_{\bar jk\bar l}
			+ \partial_{\bar j} (g^{k\bar l})\varphi_{ik\bar l}\right)
			\\
			&
			+ s g^{i\bar{j}}_\varphi g^{k\bar{q}}_\varphi g^{p\bar{l}}_\varphi(\partial_i g_{p\bar{q}} + \varphi_{ip\bar{q}})(\partial_{\bar{j}}g_{k\bar{l}} + \varphi_{\bar{j}k\bar{l}})   +
			sg^{i\bar j}\partial_i\partial_{\bar j}(g^{k\bar l})\varphi_{k\bar l} \\
			&- (1-s)\left[g^{i\bar{j}}_\varphi - g^{i\bar{j}}\right]\chi_{i\bar{j}} - (1-s)g^{i\bar{j}}\chi_{i\bar{j}} - \underline{R}^s.
		\end{split}
	\end{align}

	\emph{Step 2: Estimates for $F(\varphi)$.}

	Using the matrix identity:
	\[
	g^{i\bar{j}}_\varphi - g^{i\bar{j}} = -g^{i\bar{l}}\varphi_{k\bar{l}}g^{k\bar{j}}_\varphi,
	\]
	we obtain:
	\begin{align*}
		&g^{i\bar{j}}g^{k\bar{l}} - g^{i\bar{j}}_\varphi g^{k\bar{l}}_\varphi \\
		&= g^{i\bar{j}}(g^{k\bar{l}}-g^{k\bar{l}}_\varphi) +(g^{i\bar{j}}-g^{i\bar{j}}_\varphi)g^{k\bar{l}}_{\varphi}  \\
		&= g^{i\bar{j}}g^{k\bar{s}}\varphi_{r\bar{s}}g^{r\bar{l}}_\varphi + g^{i\bar{q}}\varphi_{p\bar{q}}g^{p\bar{j}}_\varphi g^{k\bar{l}}_\varphi .
	\end{align*}
	Thus the principal part becomes:
	\[
	s\left[g^{i\bar{j}}g^{k\bar{l}} - g^{i\bar{j}}_\varphi g^{k\bar{l}}_\varphi\right]\varphi_{i\bar{j}k\bar{l}} = s\left[g^{i\bar{j}}g^{k\bar{s}}\varphi_{r\bar{s}}g^{r\bar{l}}_\varphi + g^{i\bar{q}}\varphi_{p\bar{q}}g^{p\bar{j}}_\varphi g^{k\bar{l}}_\varphi\right]\varphi_{i\bar{j}k\bar{l}}  .
	\]
	From the above expansion, we obtain the pointwise estimate:
	\begin{align*}
		|F(\varphi)|_{c^{k-2,\alpha}(M)} \leq &C_1\left[s|\varphi|_{c^{k,\alpha}}|\varphi|_{c^{k+2,\alpha}} + s(|\varphi|_{c^{k+1,\alpha}}+|\varphi|_{c^{k+1,\alpha}}^2+|\varphi|_{c^{k,\alpha}}) \right. \\
		&\left.   + (1-s)|\varphi|_{c^{k,\alpha}} + 1\right],
	\end{align*}
	where $C_1 = C_1(\chi,\omega) > 0$ depends on the background metric and twisting form.

	\emph{Step 3: Verification of $v \in V_x(J)$.}

	For $\varphi \in V_x(J)$, we have the bounds(see \eqref{eq:cond1}, \eqref{eq:cond2}):
	\begin{align*}
		|\varphi(t)|_{c^{k,\alpha}} \leq \epsilon_0, \quad
		t^{1/2}|\varphi(t)|_{c^{k+2,\alpha}} \leq \epsilon_0.
	\end{align*}
	Combining the interpolation inequality
	\begin{align}\label{holderinterpolation}
		|\varphi|^2_{c^{k+1,\alpha}} \leq c_3|\varphi|_{c^{k+2,\alpha}}|\varphi|_{c^{k,\alpha}},
	\end{align} where $c_3$ depends only on $ \omega $, we have
	\begin{align}\label{F:E0norm}
		\begin{split}
			t^{1/2}|F(\varphi(t))|_{c^{k-2,\alpha}} \leq &C_2t^{1/2}\left[s|\varphi|_{c^{k,\alpha}}|\varphi|_{c^{k+2,\alpha}} +  s|\varphi|^{1/2}_{c^{k+2,\alpha}}|\varphi|^{1/2}_{c^{k,\alpha}}+|\varphi|_{c^{k,\alpha}}+1  \right]\\
			\leq &C_2\left[s\epsilon_0^2 +s\epsilon_0 T^{1/4} + (  1 + \epsilon_0)  T^{1/2}\right]
			,
		\end{split}
	\end{align}
	where $C_2$ depends on $ \omega$ and $\chi$.
	Since $\epsilon_0 = 2(c_1+1)\epsilon$ and we can choose $\epsilon$, $T$ small, we conclude $F(\varphi) \in E_0(J)$.

	The linear operator $K: E_0(J) \to E_1(J)$ satisfies (see \cite{MR1838320}):
	\[
	|K(F)|_{E_1(J)} \leq c_4|F|_{E_0(J)},
	\]
	where $c_4$ depends only on $ \omega $.
	For the semigroup part, from \eqref{equa:3.3} in  \cref{lem:nonempty} we have:
	\[
	|e^{-t\Delta^2}x|_{E_1(J)} \leq c_1|x|_{E_{1/2}} \leq c_1\epsilon < \frac{\epsilon_0}{2}.
	\]
	Therefore:
	\[
	|v|_{E_1(J)} \leq |e^{-t\Delta^2}x|_{E_1(J)} + \frac{1}{s}|K(F)|_{E_1(J)} \leq \frac{\epsilon_0}{2} + \frac{c_4}{s}|F|_{E_0(J)}.
	\]
	According to \eqref{F:E0norm}, choosing $\epsilon_0<1$ small enough such that
	\begin{align*}
		c_4 C_2\epsilon_0<\frac{1}{8}, \quad \text{i,e,} \quad \epsilon<\frac{1}{16c_4C_2(1+c_1)}
	\end{align*}
	and $T=C_3\epsilon_0^2s^2$ with $C_3$ satisfying
	\[
	c_4 C_2\left[   C_3^{1/4} + 2C_3^{1/2}\right]\leq \frac{1}{8},
	\]
	we obtain
	\begin{align}\label{FE0estimate}
		\frac{c_4}{s}|F|_{E_0(J)}<\frac{\epsilon_0}{4}.
	\end{align}
	It follows	$|v|_{E_1(J)} < \epsilon_0$.

	For the $C([0,T], E_{1/2})$ bound, we use the embedding estimate:
	\[
	|v|_{C(J,E_{1/2})} \leq |e^{-t\Delta^2}x|_{C(J,E_{1/2})} + \frac{1}{s}|K(F)|_{C(J,E_{1/2})}.
	\]
	From the properties of analytic semigroups \eqref{equa:3.4},
	\[
	|e^{-t\Delta^2}x|_{C(J,E_{1/2})} \leq |x|_{E_{1/2}} + \frac{\epsilon_0}{4} \leq \epsilon + \frac{\epsilon_0}{4}<\frac{3\epsilon_0}{4},
	\]
	and
	\[
	|K(F)|_{C(J,E_{1/2})} \leq c_4|F|_{E_0(J)},
	\]
	we have
	\[
	|v|_{C(J,E_{1/2})} \leq \frac{3\epsilon_0}{4} + \frac{c_4}{s}|F|_{E_0(J)}.
	\]
	According to \eqref{FE0estimate}:
	\[
	\frac{3\epsilon_0}{4} + \frac{c_4}{s}|F|_{E_0(J)} <\epsilon_0,
	\]
	we obtain the desired bound. 	Therefore, $v \in V_x(J)$, completing the proof.
\end{proof}

\begin{lemma}[Contraction property]\label{lem:contraction}
	For sufficiently small $T$ and $\epsilon$, the map $\Pi$ is a contraction on $V_x(J)$.
\end{lemma}

\begin{proof}
	We establish the contraction estimate through careful analysis of the nonlinear differences.	For $x,y\in B_{E_{1/2}}(0,\epsilon)$, let  $v_1 \in V_{x}(J),v_2\in V_{y}(J)$ to \eqref{tcfdef} with initial data $x, y$ and $F(\varphi), F(\psi)$. It is clear that
	\begin{align}\label{eq23}
		|v_1-v_2|_{E_1(J)}\leq | e^{-t\Delta^2}x- e^{-t\Delta^2}y|_{E_1(J)}+\frac{1}{s} |K(F(\varphi)-F(\psi)) |_{E_1(J)}.
	\end{align}
	According to \eqref{express:F}, we have
	\begin{align*}
		F(\varphi) - F(\psi) = &sg^{i\bar j} g^{k\bar l} (\varphi_{i\bar jk\bar l}-\psi_{i\bar jk\bar l})
		-s( g^{i\bar j}_{\varphi} g^{k\bar l}_{\varphi}  \varphi_{i\bar jk\bar l}
		-
		g^{i\bar j}_{\psi} g^{k\bar l}_{\psi}  \psi_{i\bar jk\bar l})
		\\
		&
		+sg^{i\bar j} (\partial_ig^{k\bar l}) (\varphi_{\bar jk\bar l}-\psi_{\bar jk\bar l})
		+sg^{i\bar j} (\partial_{\bar j}g^{k\bar l}) (\varphi_{ik\bar l}-\psi_{ik\bar l})
		\\
		&+sg^{i\bar j} \partial_i\partial_{\bar j}(g^{k\bar l})	 (\varphi_{k\bar l}-\psi_{k\bar l})
		+
		sg^{i\bar j}_{\varphi} g^{k\bar q}_{\varphi} g^{p\bar l}_{\varphi}(
		\partial_{i}g_{p\bar q}+ \varphi_{ip\bar q})(
		\partial_{\bar j}g_{k\bar l}+ \varphi_{\bar jk\bar l} )
		\\
		&
		-sg^{i\bar j}_{\psi} g^{k\bar q}_{\psi} g^{p\bar l}_{\psi}(
		\partial_{i}g_{p\bar q}+ \psi_{ip\bar q})(
		\partial_{\bar j}g_{k\bar l}+ \psi_{\bar jk\bar l} )
		\\
		&
		-
		s(g^{i\bar j}_{\varphi} g^{kl}_{\varphi}-g^{i\bar j}_{\psi} g^{kl}_{\psi}) \partial_i\partial_{\bar j}g_{k\bar l}
		-(1-s) (g^{i\bar j}_{\varphi}-g^{i\bar j}_{\psi})\chi_{i\bar j}
	\end{align*}
	We analyze term by term.
	By
	$
	g^{i\bar j}_{\varphi }-g^{i\bar j}_{\psi}=g_{\varphi}^{i\bar q} (\psi_{p\bar q}-\varphi_{p\bar q})g_{\psi}^{p\bar j},
	$
	the 4th derivative terms can be be treated as follows,
	\begin{align*}
		&g^{i\bar j}_{\varphi } g^{k\bar l}_{\varphi } \varphi_{i\bar jk\bar l}
		-
		g^{i\bar j}_{\psi} g^{k\bar l}_{\psi} \psi_{i\bar jk\bar l}
		\\
		=&g^{i\bar j}_{\varphi } g^{k\bar l}_{\varphi } \varphi_{i\bar jk\bar l}
		-
		g^{i\bar j}_{\psi} g^{k\bar l}_{\varphi } \varphi_{i\bar jk\bar l}
		+
		g^{i\bar j}_{\psi} g^{k\bar l}_{\varphi } \varphi_{i\bar jk\bar l}
		-g^{i\bar j}_{\psi} g^{k\bar l}_{\psi}  \varphi_{i\bar jk\bar l}
		+g^{i\bar j}_{\psi} g^{k\bar l}_{\psi} \varphi_{i\bar jk\bar l}
		-
		g^{i\bar j}_{\psi} g^{k\bar l}_{\psi}  \psi_{i\bar jk\bar l}
		\\
		=& g^{k\bar l}_{\varphi } g_{\varphi }^{i\bar q}(\psi_{p\bar q}-\varphi_{p\bar q})g_{\psi}^{p\bar j}\varphi_{i\bar jk\bar l}
		+
		g^{i\bar j}_{\psi} g_{\varphi }^{k\bar q}(\psi_{p\bar q}-\varphi_{p\bar q})g_{\psi}^{p\bar l} \varphi_{i\bar jk\bar l}
		+g^{i\bar j}_{\psi} g^{k\bar l}_{\psi} (\varphi_{i\bar jk\bar l} -\psi_{i\bar jk\bar l})
	\end{align*}
	and
	\begin{align*}
		&g^{i\bar j} g^{k\bar l} (\varphi_{i\bar jk\bar l}-\psi_{i\bar jk\bar l})
		- ( g^{i\bar j}_{\varphi} g^{k\bar l}_{\varphi}  \varphi_{i\bar jk\bar l}
		-
		g^{i\bar j}_{\psi} g^{k\bar l}_{\psi}  \psi_{i\bar jk\bar l})
		\\
		=&(g^{i\bar q}\psi_{p\bar q}g^{p\bar j}g_\psi^{k\bar l}
		+
		g^{i\bar j}g^{k\bar q}\psi_{p\bar q}g^{p\bar l})
		(\varphi_{i\bar jk\bar l} -\psi_{i\bar jk\bar l})
		\\
		&-g^{k\bar l}_{\varphi } g_{\varphi }^{i\bar q}(\psi_{p\bar q}-\varphi_{p\bar q})g_{\psi}^{p\bar j}\varphi_{i\bar jk\bar l}
		-
		g^{i\bar j}_{\psi} g_{\varphi }^{k\bar q}(\psi_{p\bar q}-\varphi_{p\bar q})g_{\psi}^{p\bar l} \varphi_{i\bar jk\bar l}.
	\end{align*}
	The 3rd derivative terms can be treated as follows,
	\begin{align*}
		& g^{i\bar j}_{\varphi} g^{k\bar q}_{\varphi} g^{p\bar l}_{\varphi}(
		\partial_{i}g_{p\bar q}+ \varphi_{ip\bar q})(
		\partial_{\bar j}g_{k\bar l}+ \varphi_{\bar jk\bar l} )
		- g^{i\bar j}_{\psi} g^{k\bar q}_{\psi} g^{p\bar l}_{\psi}(
		\partial_{i}g_{p\bar q}+ \psi_{ip\bar q})(
		\partial_{\bar j}g_{k\bar l}+ \psi_{\bar jk\bar l} )
		\\
		=&(g^{i\bar j}_{\varphi} g^{k\bar q}_{\varphi} g^{p\bar l}_{\varphi}-g^{i\bar j}_{\psi} g^{k\bar q}_{\psi} g^{p\bar l}_{\psi})(
		\partial_{i}g_{p\bar q}+ \varphi_{ip\bar q})(
		\partial_{\bar j}g_{k\bar l}+ \varphi_{\bar jk\bar l} )
		\\
		&+
		g^{i\bar j}_{\psi} g^{k\bar q}_{\psi} g^{p\bar l}_{\psi}(
		\partial_{i}g_{p\bar q}+ \varphi_{ip\bar q})(
		\varphi_{\bar jk\bar l} -\psi_{\bar jk\bar l} )
		+
		g^{i\bar j}_{\psi} g^{k\bar q}_{\psi} g^{p\bar l}_{\psi}(
		\varphi_{ip\bar q}-\psi_{ip\bar q})(
		\partial_{\bar j}g_{k\bar l}+ \psi_{\bar jk\bar l} ).
	\end{align*}
	Hence
	\begin{align}
		\label{estimateofF0}
		\begin{split}
			&t^{1/2}|F(\varphi )-F(\psi)|_{c^{k-2,\alpha}}\\
			\leq
			& C_4t^{1/2}\left(s|\varphi -\psi|_{c^{k+2,\alpha}}|\psi|_{c^{k,\alpha}}+s|\varphi -\psi|_{c^{k,\alpha}}|\varphi |_{c^{k+2,\alpha}}\right)
			\\
			&+C_4t^{1/2}\left(s|\varphi-\psi|_{c^{k,\alpha}}|\varphi|^2_{c^{k+1,\alpha}}+
			s|\varphi-\psi|_{c^{k+1,\alpha}}(|\varphi|_{c^{k+1,\alpha}}
			+|\psi|_{c^{k+1,\alpha}})\right)
			\\
			&+ C_4t^{1/2}\left(s|\varphi-\psi|_{c^{k+1,\alpha}}+s|\varphi-\psi|_{c^{k,\alpha}}+(1-s)|\varphi-\psi|_{c^{k,\alpha}} \right).
		\end{split}
	\end{align}
	Since $\varphi \in V_{x}(J),\psi\in V_{y}(J)$ imply that $\varphi,\psi$ both satisfy \eqref{eq:cond1},\eqref{eq:cond2}, we have
	\begin{align}
		\label{extiamteofF1}
		\begin{split}
			&t^{1/2}|\varphi -\psi|_{c^{k+2,\alpha}}|\psi|_{c^{k,\alpha}}\leq \epsilon_0|\varphi-\psi|_{E_1(J)}, \\
			&
			t^{1/2}|\varphi -\psi|_{c^{k,\alpha}}|\varphi |_{c^{k+2,\alpha}}\leq \epsilon_0|\varphi-\psi|_{C(J,E_{1/2})},
			\\
			&t^{1/2}|\varphi-\psi|_{c^{k,\alpha}}|\varphi|^2_{c^{k+1,\alpha}}\leq c_3\epsilon_0^2|\varphi-\psi|_{C(J,E_{1/2})}.
		\end{split}
	\end{align}
	The interpolation inequality \eqref{holderinterpolation} implies
	\begin{align*}
		t^{1/4}|\varphi-\psi|_{c^{k+1,\alpha}}
		\leq& c^{1/2}_3\left(t^{1/4}|\varphi-\psi|^{1/2}_{c^{k,\alpha}}|\varphi-\psi|^{1/2}_{c^{k+2,\alpha}}\right)\\
		\leq& c^{1/2}_3/2\left(|\varphi-\psi|_{c^{k,\alpha}}+t^{1/2}|\varphi-\psi|_{c^{k+2,\alpha}}\right)
	\end{align*}
	and
	\begin{align*}
		t^{1/4}|\varphi|_{c^{k+1,\alpha}}\leq c^{1/2}_3|\varphi |^{1/2}_{c^{k,\alpha}}t^{1/4}|\varphi |^{1/2}_{c^{k+2,\alpha}}\leq c^{1/2}_3\epsilon_0.
	\end{align*}
	It follows
	\begin{align}\label{extiamteofF2}
		&t^{1/2}|\varphi-\psi|_{c^{k+1,\alpha}}(|\varphi|_{c^{k+1,\alpha}}
		+|\psi|_{c^{k+1,\alpha}})\leq c_3\epsilon_0 (|\varphi-\psi|_{C(J,E_{1/2})}+ |\varphi-\psi|_{E_1(J)} ),
	\end{align}
	and
	\begin{align}
		\label{extiamteofF3}
		\begin{split}
			&C_4t^{1/2}\left(s|\varphi-\psi|_{c^{k+1,\alpha}}+s|\varphi-\psi|_{c^{k+1,\alpha}}+(1-s)|\varphi-\psi|_{c^{k+1,\alpha}} \right)\\
			&\leq C_5(T^{1/2}+T^{1/4}s) |\varphi-\psi|_{C(J,E_{1/2})}+C_6T^{1/4}s|\varphi-\psi|_{E_1(J)}.
		\end{split}
	\end{align}
	According to \eqref{estimateofF0},\eqref{extiamteofF1},\eqref{extiamteofF2} and \eqref{extiamteofF3}, if we assume $\epsilon_0\leq 1, T\leq 1$, we have
	\begin{align}\label{fu-fv:estimate}
		\begin{split}
			|F(\varphi)-F(\psi)|_{E_0(J)}\leq C_7 (\epsilon_0+T^{1/4})s|\varphi-\psi|_{E_1(J)}\\
			+C_8(\epsilon_0s +T^{1/2}+T^{1/4}s)|\varphi-\psi|_{C(J,E_{1/2})}.
		\end{split}
	\end{align}
	Since $\varphi-\psi-e^{-t\Delta^2}(x-y)|_{t=0}=0$, by \eqref{eq:3.4}, we estimate
	\begin{align*}
		|\varphi-\psi|_{C(J,E_{1/2})}
		\leq& |\varphi-\psi-e^{-t\Delta^2}(x-y)|_{C(J,E_{1/2})}
		+|e^{-t\Delta^2}(x-y)|_{C(J,E_{1/2})}
		\\
		\leq &c_2|\varphi-\psi-e^{-t\Delta^2}(x-y)|_{E_1(J)}
		+|e^{-t\Delta^2}(x-y)|_{C(J,E_{1/2})}
		\\
		\leq & c_2|\varphi-\psi|_{E_1(J)}+c_2|e^{-t\Delta^2}(x-y)|_{E_1(J)}
		\\
		&+|e^{-t\Delta^2}(x-y)|_{C(J,E_{1/2})}.
	\end{align*}
	According to \eqref{equa:3.3}, we have
	\begin{align}\label{eq24}
		|e^{-t\Delta^2}(x-y)|_{E_1(J)}\leq c_1|x-y|_{E_{1/2}}.
	\end{align}
	We also have $|e^{-t\Delta^2}(x-y)|_{C(J,E_{1/2})}\leq c_3|x-y|_{E_{1/2}}$. It follows that
	\begin{align}\label{var1-var2}
		|\varphi-\psi|_{C(J,E_{1/2})}
		\leq &  c_1|\varphi-\psi|_{E_1(J)}+(c_1c_2+c_3)|x-y|_{E_{1/2}} .
	\end{align}
	According to \eqref{eq23}, \eqref{fu-fv:estimate} and \eqref{var1-var2}, we have
	\begin{align}\label{equa:23}
		\begin{split}
			&|v_1-v_2|_{E_1(J)}\\
			\leq &c_1|x-y|_{E_{1/2}}+\frac{1}{s}\|K\| |F(\varphi)-F(\psi)|_{E_0(J)}
			\\
			\leq&  C_9(\epsilon_0  +T^{1/4}+T^{1/2}/s)(|\varphi-\psi|_{E_1(J)}+|x-y|_{E_{1/2}}).
		\end{split}
	\end{align}
	The positive constants
	$C_1, \ldots, C_9$ appearing above  depend only on $\omega,\chi$. Now we choose
	$$
	\epsilon_0=\frac{1}{6C_9},\quad T=\frac{s^2}{(6C_9+1)^4},
	$$
	we obtain
	\begin{align}\label{equa:22}
		&|v_1-v_2|_{E_1(J)}
		\leq   \frac{1}{2}|x-y|_{E_{1/2}}+\frac{1}{2} | \varphi-\psi|_{E_1(J)}.
	\end{align}
	In particular, if we take $y=x$, then we have
	\begin{align*}
		|v_1-v_2|_{E_1(J)}\leq  \frac{1}{2} | \varphi-\psi|_{E_1(J)}.
	\end{align*}
	Hence $\Pi$ is a contraction map with constant $1/2$ and  has a unique fixed point $\varphi(t,x)\in V_x(J)$ fo reach $x\in B_{E_{1/2}}(0,\epsilon)$. $v\in C((0,T), C^\infty(M))$ implies that the solution $\varphi(t)\in C((0,T), C^\infty(M))$. By \eqref{equa:22},  we have
	\begin{align}\label{E1normcont}
		|\varphi(t,x)-\varphi(t,y)|_{E_1(J)}\leq   |x-y|_{E_{1/2}}.
	\end{align}
	Denote $\theta(t)=\varphi(t,x)-\varphi(t,y)$. Since $(\theta(t)-e^{-t\Delta^2}(x-y))|_{t=0}=0$, applying \eqref{eq:3.4} yields
	\begin{align*}
		|\theta(t)|_{C([0,T], E_{1/2})}\leq& |e^{-t\Delta^2}(x-y)|_{C([0,T], E_{1/2})}+c_2|\theta(t)-e^{-t\Delta^2}(x-y)|_{E_1(J)}
		\\
		\leq &|x-y|_{E_{1/2}}+c_2|\theta(t)|_{E_1(J)}+c_2|e^{-t\Delta^2}(x-y)|_{E_1(J)}\\
		\leq &(1+c_1c_2)|x-y|_{E_{1/2}}+c_2|\theta(t)|_{E_1(J)},
	\end{align*}
	where the last inequality follows from \eqref{equa:3.3}.
	Combining this with \eqref{E1normcont} gives \eqref{eq:stability2}.
\end{proof}

\subsection{Uniqueness and continuous dependence}

\begin{lemma}[Uniqueness]\label{lem:uniqueness}
	The solution provided by the fixed point theorem is unique in $V_x(J)$.
\end{lemma}

\begin{proof}
	Suppose $u_1, u_2$ are two solutions in $V_x(J^*)$ for some $J^* = [0,T^*]$. Let:
	\[
	T_1 = \sup\{t \in [0,T] : u_1(\tau) = u_2(\tau)\ \forall \tau \in [0,t)\}.
	\]
	Since the fixed point is unique in $V_x(J^*)$ for small $T^*$, we have $T_1 > 0$. If $T_1 < T$, then $u_1(T_1) = u_2(T_1) = y$, and both $u_1(t+T_1)$, $u_2(t+T_1)$ solve the equation with initial data $y$, contradicting local uniqueness. Hence $T_1 = T$.
\end{proof}

The continuous dependence estimates \eqref{eq:stability1} and \eqref{eq:stability2} follow from similar estimates on the solution map.

\begin{remark}[Dependence on parameters]
	The initial neighborhood radius $\epsilon$ only depends on $\omega, \chi$ and is independent of the twisted weight $s$. The time interval $T$ depends on the twisting parameter $s$. This reflects the scaling properties of the fourth-order parabolic equation.
\end{remark}

%
%
%
%
%

This establishes the foundation for our subsequent analysis of the long-time behavior and stability properties of the twisted Calabi flow.

The short-time existence of the Calabi flow was first proved by Chen and He \cite{MR2405167} for small initial data
 $\varphi_0\in c^{3,\alpha}(M)$.
He  later extended the result to  small
$\varphi_0\in c^{2,\alpha}(M)$ using analytic semigroup theory and a contraction mapping argument in \cite{MR3010280}.
Subsequently, He and Zeng \cite{MR4259154} established short-time existence under the weaker condition
$\partial\bar\partial\varphi_0\in L^\infty(M)$ and $(1-\delta)\omega<\omega_{\varphi_0}<(1+\delta)\omega$.
\begin{rem}
	In \cite{MR2405167,MR3010280}, the authors work in the little H\"older space $c^{k,\alpha}$
	rather than the usual H\"older space $C^{k,\alpha}(M)$.
	Since smooth functions are not dense in $C^{k,\alpha}(M)$, one cannot directly approximate general elements of $C^{k,\alpha}(M)$ by smooth ones.
	However, the fundamental solution to the biharmonic heat flow can be approximated by smooth functions, and hence belongs to the little H\"older space
$c^{k,\alpha}(M)$.
\end{rem}

\section{Stability analysis near twisted cscK metrics}
\label{sec:stability-analysis}

This section establishes the fundamental stability property of the twisted Calabi flow in the vicinity of twisted constant scalar curvature K\"ahler metrics. We prove that initial data sufficiently close to a twisted cscK metric yield solutions that exist for all time and converge back to the original metric.

\subsection{Stability framework and preliminary estimates}

We begin by formulating the stability problem in appropriate function spaces. Recall the twisted Calabi flow equation:
\begin{align*}
	\frac{\partial \varphi}{\partial t} = R^s(\omega_\varphi) - \underline{R}^s, \quad \varphi(0) = \varphi_0.
\end{align*}

Assume $\omega$ is a twisted cscK metric, so $R^s(\omega) \equiv \underline{R}^s$. We study the behavior of solutions starting from nearby initial data.

\begin{defn}[Stability neighborhoods]
	For $\delta>0$, $1>\lambda > 0$, define the neighborhoods:
	\begin{align*}
		\mathcal{V}^{k,\alpha}_{\delta,\lambda} &= \left\{ \varphi \in c^{k,\alpha}(M) : \lambda\omega < \omega_\varphi < \lambda^{-1}\omega,\ \|\varphi\|_{c^{k,\alpha}(M)} \leq \delta \right\}, \\
		B^{k,\alpha}_{\delta} &= \left\{ \varphi \in c^{k,\alpha}(M) : \|\varphi\|_{c^{k,\alpha}(M)} \leq \delta \right\}.
	\end{align*}
\end{defn}

From the short-time existence theorem (\cref{thm:short-time-existence}), there exist some $\epsilon>0$ and $T>0$ such that for any $\varphi_0 \in B^{2,\alpha}_{\epsilon}$, the twisted Calabi flow admits a unique solution:
\[
\varphi(t, \varphi_0) \in E_1([0,T]) \cap C([0,T], c^{2,\alpha}(M)),
\]
satisfying the estimates:
\begin{align} \label{stability-estimates}
	\|\varphi\|_{C([0,T], c^{2,\alpha}(M))} &\leq c\|\varphi_0\|_{c^{2,\alpha}(M)}, \\
	t^{1/2}(\|\dot\varphi(t)\|_{ c^{\alpha}(M)} +\|\varphi(t)\|_{ c^{4,\alpha}(M)} )&\leq c\|\varphi_0\|_{c^{2,\alpha}(M)}\label{stability-estimates2}.
\end{align}

\subsection{Spectral gap and energy decay}

The cornerstone of our stability analysis is a uniform spectral gap estimate for the twisted Lichnerowicz operator along the flow.

\begin{lem}[Uniform spectral gap] \label{lem:uniform-spectral-gap}
	If
	\[
	\chi(X,JX) \geq \kappa\omega(X,JX), \quad \forall X \in \Gamma(TM),
	\]
	for some $\kappa > 0$, and $\varphi\in \mV^{2,\alpha}_{\epsilon, \lambda}$, then there exists some $\lambda_0=\lambda_0(\omega, \lambda,\kappa,\epsilon)>0$ such that the first eigenvalue $\lambda_1(\varphi)$ of the twisted Lichnerowicz operator $\bL^s_\varphi$ satisfies:
	\[
	\lambda_1(\varphi) \geq \lambda_0=(1-s)C(\omega, \lambda,\kappa,\epsilon).
	\]
\end{lem}

\begin{proof}
	Under the metric equivalence $\lambda\omega \leq \omega_\varphi \leq \lambda^{-1}\omega$, we have:
	\begin{align*}
		\lambda^m \omega^m \leq \omega_\varphi^m \leq \lambda^{-m} \omega^m, \quad
		\lambda^{-1} |\nabla f|^2\geq |\nabla_\varphi f|_\varphi^2 \geq \lambda |\nabla f|^2.
	\end{align*}
	Denote $c_\omega(f)=(\int_Mf\omega^m)/(\int_M\omega^m), c_\varphi(f)=(\int_Mf\omega_\varphi^m)/(\int_M\omega_\varphi^m)$.
	We have
	\begin{align*}
	   \left|\int_Mf(\omega^m-\omega_\varphi^m)\right|
&	= \left|\int_Mfi\ddbar\varphi\wedge(\sum_{i=0}^{m-1}\omega^i\wedge\omega^{m-1-i}_\varphi)\right|
	\\
		\leq  &C(\lambda)\left|\int_M i\partial f\wedge\bar\partial\varphi\wedge \omega^{m-1}_\varphi \right|
		\leq C(\lambda,\omega,\epsilon) \int_M|\nabla_\varphi  f|_\varphi  \omega^m_\varphi.
	\end{align*}
	It follows that
	\begin{align*}
	\int_M (c_\omega(f)-c_\varphi(f))^2 \omega_\varphi^m
	\leq &C(\lambda,\omega,\epsilon)\left(\int_M|\nabla_\varphi  f|_\varphi  \omega^m_\varphi\right)^2 /\int_M\omega^m_\varphi
	\\
	\leq & C(\lambda,\omega,\epsilon) \int_M|\nabla_\varphi  f|^2_\varphi  \omega^m_\varphi .
	\end{align*}
	Let $\mu_1$ be the first positive eigenvalue of the Laplacian $\Delta_\omega$. We have
	\begin{align*}
		\int_M (f-c_\varphi(f))^2 \omega_\varphi^m\leq &
		2\int_M (f-c_\omega(f))^2 \omega_\varphi^m +2\int_M (c_\omega(f)-c_\varphi(f))^2 \omega_\varphi^m
		\\
		\leq
		&
		2\lambda^{-m}\int_M (f-c_\omega(f))^2 \omega^m + C(\lambda,\omega,\epsilon)\int_M|\nabla_\varphi f|_\varphi^2\omega^m
		\\
		\leq
		&
		\frac{2\lambda^{-m}}{\mu_1}\int_M|\nabla f|^2\omega^m
		+ C(\lambda,\omega,\epsilon)\int_M|\nabla_\varphi f|_\varphi^2\omega^m
		\\
		\leq
		&
		\left(\frac{2\lambda^{-2m-1}}{\mu_1}+ C(\lambda,\omega,\epsilon)\right)\int_M|\nabla_\varphi f|_\varphi^2\omega^m_\varphi.
	\end{align*}
	The first eigenvalue of the Laplacian $\Delta_{\omega_\varphi}$ satisfies:
	\begin{align*}
		\mu_1(\varphi) &=\inf_{\substack{f \in C^\infty(M,\mathbb{R}),\\ f\neq const}} \frac{\int_M |\nabla_\varphi f|_\varphi^2 \omega_\varphi^m}{\int_M (f-c_\varphi(f))^2 \omega_\varphi^m} \geq \frac{1}{ \frac{2\lambda^{-2m-1}}{\mu_1}+ C(\lambda,\omega,\epsilon)}.
	\end{align*}
	Moreover, the curvature bound implies $\chi \geq \kappa\lambda \omega_\varphi$. Applying  \cref{lem:1.6} to the metric $\omega_\varphi$ yields:
	\[
	\lambda_1(\varphi) \geq (1-s)\kappa\lambda \mu_1(\varphi) \geq \lambda_0:=
	\frac{(1-s)\kappa\lambda}{ \frac{2\lambda^{-2m-1}}{\mu_1}+ C(\lambda,\omega,\epsilon)} .
	\]
\end{proof}

This spectral gap enables us to establish exponential decay of the twisted Calabi energy.

\begin{prop}[Energy decay] \label{prop:energy-decay}
	Let $\varphi(t)$ be a solution of the twisted Calabi flow with $ \varphi(t) \in \mathcal{V}^{4,\alpha}_{\epsilon,\lambda}$ for $t \in [0,T]$. Then the twisted Calabi energy decays exponentially:
	\begin{align} \label{energy-decay}
		\mC^s(\varphi(t)) \leq \mC^s(\varphi(0)) e^{-\lambda_0 t},
	\end{align}
	where $\lambda_0$ is defined in \cref{lem:uniform-spectral-gap}.
\end{prop}

\begin{proof}
	Along the twisted Calabi flow, we compute:
	\begin{align*}
		\frac{d}{dt} \mC^s(\varphi(t))
		&= -\int_M (R^s(\varphi) - \underline{R}^s) \bL^s_\varphi(R^s(\varphi) - \underline{R}^s) \omega_\varphi^m \\
		&\leq -\lambda_1(\varphi) \int_M (R^s(\varphi) - \underline{R}^s)^2 \omega_\varphi^m \\
		&\leq -\lambda_0 \mC^s(\varphi(t)).
	\end{align*}
	Gronwall's inequality yields the exponential decay.
\end{proof}

\subsection{A priori estimates}

To establish global existence, we derive a priori estimates that control the geometry along the flow.

\begin{lem}[Elliptic regularity estimate] \label{lem:elliptic-regularity}
	Let $\varphi \in \mathcal{V}^{4,\alpha}_{\epsilon,\lambda}$. Then for some $p > 2m$, there exists $C = C(\lambda, \epsilon, p) > 0$ such that:
	\begin{align} \label{elliptic-estimate}
			\|\varphi(t)\|_{C^{3,\alpha}(M)} \leq   C\left( \frac{1-s}{s} \|R^s - \underline{R}^s\|_{L^p(M,g)} + \frac{1}{s^{3p}} \|\varphi\|_{L^1(M,g)} \right).
	\end{align}
\end{lem}

\begin{proof}
	We reformulate the twisted scalar curvature equation as an elliptic problem. Recall:
	\[
	R^s - \underline{R}^s = sR_\varphi - (1-s)\tr_\varphi\chi - \underline{R}^s =: f.
	\]
Rewriting this equation:
	\begin{align*}
		-g^{i\bar{j}}_\varphi \partial_i\partial_{\bar{j}} \log\frac{\det(g_{k\bar{l}} + \varphi_{k\bar{l}})}{\det(g_{k\bar{l}})}
		&= (g^{i\bar{j}}_\varphi - g^{i\bar{j}})\partial_i\partial_{\bar{j}}\log\det(g_{k\bar{l}}) \\
		&\quad + \frac{1-s}{s}(g^{i\bar{j}}_\varphi - g^{i\bar{j}})\chi_{i\bar{j}} + \frac{1-s}{s}f.
	\end{align*}
Define $u = \log\frac{\det(g_{i\bar{j}} + \varphi_{i\bar{j}})}{\det(g_{i\bar{j}})}$ and
	\begin{align}\label{expression:f}
	h_s = (g^{i\bar{j}}_\varphi - g^{i\bar{j}})\partial_i\partial_{\bar{j}}\log\det(g_{k\bar{l}}) + \frac{1-s}{s}(g^{i\bar{j}}_\varphi - g^{i\bar{j}})\chi_{i\bar{j}} + \frac{1-s}{s}f.
	\end{align}
Then we have the uniform elliptic equation:
	\begin{align} \label{elliptic-equation}
		-\Delta_\varphi u = h_s.
	\end{align}
	For any $p>1$,
	$$\Delta_\varphi:\{\phi\in W^{2,p}(\varphi):\int_M\phi \omega^m_\varphi=0\}\to \{\phi\in L^p(\varphi):\int_M\phi \omega^m_\varphi=0\}$$
	is an invertible operator. We have
	\begin{align*}
		\|u - \underline{u}\|_{W^{2,p}(M,g_\varphi)} \leq C(\lambda, \epsilon,p) \|h_s\|_{L^p(M,g_\varphi)},
	\end{align*}
	where $\underline{u}=(\int_Mu\omega_\varphi^m)/(\int_M \omega_\varphi^m)$ is the average of $u$ under the volume form $\omega_\varphi^m$.
	By \cref{prop:norm-equivalence}, $W^{2,p}$-norms with respect to $g$ and $g_\varphi$ are equivalent:
	\begin{align}\label{w2pestimate}
		\|u - \underline{u}\|_{W^{2,p}(M,g)} \leq C(\lambda, p, \epsilon) \|h_s\|_{L^p(M,g)}.
	\end{align}
From the definition of $h_s$ in \eqref{expression:f} and the identity $g^{i\bar{j}}_\varphi - g^{i\bar{j}} =- g^{i\bar{q}}_\varphi \varphi_{p\bar{q}} g^{p\bar{j}}$, we estimate:
	\begin{align}\label{lpestimate}
		\|h_s\|_{L^p(M,g)} \leq C\left( \frac{1-s}{s} \|f\|_{L^p(M,g)} + \frac{1}{s} \|\varphi\|_{W^{2,p}(M,g)} \right).
	\end{align}
To recover $\varphi$ from $u$, consider the Monge-Ampere equation:
	\[
	\frac{\det(g_{i\bar{j}} + \varphi_{i\bar{j}})}{\det(g_{i\bar{j}})} = e^{\underline{u}} e^{u - \underline{u}}.
	\]
Taking a logarithm and differentiating, we  have
\begin{align*}
g^{i\bar j}_\varphi (\partial_\alpha g_{i\bar{j}} + \varphi_{\alpha i\bar{j}})-\partial_\alpha\log\det(g_{i\bar{j}})=\partial_\alpha(u-\underline{u})
\end{align*}
Viewing the above as a equation about $\varphi_\alpha$ and applying Schauder estimates (see for example \cite[Proposition 3.11]{MR3186384})  yields:
\begin{align*}
\|\varphi_\alpha\|_{C^{2,\alpha}(M)}\leq C(\lambda, \epsilon )\|\partial_\alpha(u-\underline{u})\|_{C^{\alpha}(M)},
\end{align*}
i.e.,
	\begin{align*}
		\|\varphi\|_{C^{3,\alpha}(M)} \leq C(\lambda, \epsilon ) \|u - \underline{u}\|_{C^{1,\alpha}(M)} .
	\end{align*}
Sobolev embedding $W^{2,p}(M)\hookrightarrow C^{1,\alpha}(M), \alpha=1-2m/p$ implies
\begin{align}\label{c3alphaestimate}
	\|\varphi\|_{C^{3,\alpha}(M)}  \leq C(\lambda, \epsilon ) \|u - \underline{u}\|_{W^{2,p}(M,g)}.
\end{align}
Combining \eqref{c3alphaestimate}, \eqref{w2pestimate} and \eqref{lpestimate} gives
\begin{align*}
	\|\varphi\|_{C^{3,\alpha}(M)} \leq C\left( \frac{1-s}{s} \|R^s - \underline{R}^s\|_{L^p(M,g)} + \frac{1}{s} \|\varphi\|_{W^{2,p}(M,g)} \right).
\end{align*}
According to the interpolation inequality  \cref{prop:GN-inequality}, we have
\begin{align*}
	\|\varphi(t)\|_{C^{3,\alpha}(M)}
	&\leq C\left( \frac{1-s}{s} \|R^s - \underline{R}^s\|_{L^p(M,g)} + \frac{1}{s} \|\varphi\|_{C^{3,\alpha}(M)}^{\frac{3p-1}{3p}} \|\varphi\|_{L^1(M,g)}^{\frac{1}{3p}} \right).
\end{align*}
Using Young's inequality $ab \leq \frac{1}{2} a^{3p/(3p-1)} + C b^{3p}$, we obtain
\begin{align*}
	\|\varphi(t)\|_{C^{3,\alpha}(M)} \leq \frac{1}{2} \|\varphi(t)\|_{C^{3,\alpha}(M)} + C\left( \frac{1-s}{s} \|R^s - \underline{R}^s\|_{L^p(M,g)} + \frac{1}{s^{3p}} \|\varphi\|_{L^1(M,g)} \right).
\end{align*}
Absorbing the first term and applying estimates \eqref{L1-bound} and \eqref{Lp-bound}, we have
\begin{align*}
	\|\varphi(t)\|_{C^{3,\alpha}(M)} \leq   C\left( \frac{1-s}{s} \|R^s - \underline{R}^s\|_{L^p(M,g)} + \frac{1}{s^{3p}} \|\varphi\|_{L^1(M,g)} \right).
\end{align*}

\end{proof}

%
%
%
%
%

\subsection{Global Existence and Convergence}

We now establish the main stability theorem through a bootstrap argument.

\begin{thm}[Stability of twisted Calabi flow] \label{thm:stability-main}
	Assume $0<s<1$ and $\omega$ is a twisted cscK metric. There exists $\delta_0 = C(\omega,\chi)s^{p}(1-s) > 0$, where $p>6m+2$,  such that if $\|\varphi_0\|_{C^{2,\alpha}(M)} < \delta_0$, then the twisted Calabi flow starting from $\varphi_0$ exists for all time and converges smoothly  to $0$. Moreover, $\|\varphi(t)\|_{C^{k,\alpha}(M)}$ converges to 0 exponentially as $t\to\infty$ for any $k\geq 1$.
\end{thm}

\begin{proof}

Assume that the background metric $\omega$ is a twisted cscK metric.
Let $\epsilon$ and $T$ be the existence neighborhood and time given in \cref{thm:short-time-existence}.
Consider $k=2$ in \cref{thm:short-time-existence} and denote $T = 2I$ for convenience. For any
$\delta \in (0, \epsilon)$ and any initial potential satisfying
$\|\varphi_0\|_{c^{2,\alpha}(M)} < \delta$, the corresponding short-time solution
$\varphi(t)$ exists on $[0,2I]$ and satisfies
\begin{align}\label{neighborhoodcondition}
	\begin{split}
		\|\varphi\|_{c^{2,\alpha}(M)} &\le c\,\|\varphi_0\|_{c^{2,\alpha}(M)},
		\quad \forall\, t \in [0,2I], \\[4pt]
		\|\varphi\|_{c^{4,\alpha}(M)} &\le c\,I^{-1/2}\|\varphi_0\|_{c^{2,\alpha}(M)},
		\quad \forall\, t \in [I,2I].
	\end{split}
\end{align}
Since $T=C(\omega,\chi)s^2$ and $\epsilon=\epsilon(\omega,\chi)$, we can choose $\delta = C(\omega,\chi)s$ with $C(\omega,\chi)$ small enough
 so that any twisted Calabi flow starting
from $B_{\delta}^{2,\alpha}$ exists on $[0,2I]$ and satisfies
\begin{align}\label{conditionofdelta}
	\varphi(t) \in \mathcal{V}_{\epsilon,\lambda}^{2,\alpha}, \quad
	\forall\, t \in [0,2I],
	\qquad\text{and}\qquad
	\varphi(t) \in \mathcal{V}_{\epsilon,\lambda}^{4,\alpha}, \quad
	\forall\, t \in [I,2I].
\end{align}

Let $\lambda_0$ denote the uniform spectral gap constant from
\cref{lem:uniform-spectral-gap}.

The evolution of $\varphi$ can then be written as
\[
\varphi(t)
= \varphi(I) + \int_{I}^t (R^s(\tau) - \underline{R}^s)\, d\tau,
\quad \forall\, t \in [I,2I].
\]
For $t \in [I,2I]$, taking the $L^1$-norm and using the decay of the twisted
Calabi energy, we obtain
\begin{align}
	\|\varphi(t)\|_{L^1}
	&\le \|\varphi(I)\|_{L^1}
	+ \int_{I}^t \|R^s(\tau) - \underline{R}^s\|_{L^1}\, d\tau \notag\\
	&\le \|\varphi(I)\|_{L^1}
	+ C \int_{I}^t
	\left(
	\int_M (R^s(\tau) - \underline{R}^s)^2
	\, \omega_\varphi^m
	\right)^{\!\!1/2}
	d\tau \notag\\
	&\le \|\varphi(I)\|_{L^1}
	+ C\, \mC^s(\varphi(I))^{1/2}
	\int_0^t e^{-\lambda_0 \tau/2}\, d\tau \notag\\
	&\le \|\varphi(I)\|_{L^1}
	+ \frac{2C}{\lambda_0}\, \mC^s(\varphi(I))^{1/2}.
	\label{L1-bound}
\end{align}
Since $\varphi(t)\in \mathcal{V}^{4,\alpha}_{\epsilon,\lambda}$ remains small for all $t\in [I,2I]$, we have the uniform bound $\|R^s(t)-\underline{R}^s\|_{L^\infty}\le C(\epsilon,\lambda)$.
For $t\in [I,2I]$ and any $p>2m$, it follows that
\begin{align}
	\|R^s(t)-\underline{R}^s\|_{L^p}
	&\le C(\epsilon,\lambda)\|R^s(t)-\underline{R}^s\|_{L^2}^{2/p}
	\|R^s(t)-\underline{R}^s\|_{L^\infty}^{1-2/p} \notag\\
	&\le C(\epsilon,\lambda)\mC^s(\varphi(I))^{1/p} e^{-\lambda_0 t/p}. \label{Lp-bound}
\end{align}
Applying \cref{lem:elliptic-regularity}, we obtain
\begin{align}\label{bootstrap-bound}
	\|\varphi(t)\|_{C^{3,\alpha}}
	\le C\!\left(
	\frac{1-s}{s}\mC^s(\varphi(I))^{\frac{1}{p}}e^{-\frac{\lambda_0 t}{p}}
	+\frac{1}{s^{3p}}\!\left(
	\|\varphi(I)\|_{L^1}
	+\frac{2}{\lambda_0}\mC^s(\varphi(I))^{\frac{1}{2}}
	\right)
	\right).
\end{align}
By the property \eqref{neighborhoodcondition}, we can choose a sufficiently small positive constant $\delta_0<\delta$ such that any twisted Calabi flow starting from $\varphi_0\in B_{\delta_0}^{2,\alpha}$ satisfies
\begin{align}\label{condition1ofdelta0}
	\|\varphi\|_{c^{2,\alpha}(M)} &\le \delta,
	\quad \forall\, t\in[0,2I],
	\qquad
	\|\varphi\|_{c^{4,\alpha}(M)} \le \delta,
	\quad \forall\, t\in[I,2I].
\end{align}
Moreover, since
\[
\|\varphi\|_{c^{4,\alpha}(M)}
\le C\,I^{-1/2}\|\varphi_0\|_{c^{2,\alpha}(M)}
\le C(\omega,\chi)s^{-1}\delta_0, \quad \forall\, t\in[I,2I],
\]
we have
\begin{align*}
	|R^s-\underline{R}^s|_{L^\infty}
	\le C(\omega,\chi)s^{-1}\delta_0,\quad
	\mC^s(\varphi(t))
	\le C(\omega,\chi)s^{-2}\delta_0^2,
	\quad \forall\, t\in[I,2I].
\end{align*}

Noting that $\lambda_0 = (1-s)C(\omega,\lambda,\epsilon,\chi)$, we can control the
right-hand side of \eqref{bootstrap-bound} by
\begin{align}\label{conditionofdelta2}
	\|\varphi(t)\|_{C^{3,\alpha}}
	\le &C\!\left(
	\frac{1-s}{s^{1+2/p}} \delta_0^{2/p}
	+ \frac{1}{s^{3p}}\!\left(\frac{\delta_0}{s}
	+ \frac{2\delta_0}{\lambda_0 s}\right)
	\right)\notag\\
	\le& C(\omega,\chi,\epsilon,\lambda)\frac{\delta_0}{s^{1+3p}(1-s)}.
\end{align}
We then choose $\delta_0 > 0$ sufficiently small such that the right-hand side of
\eqref{conditionofdelta2} is smaller than $\delta$, i.e.
\begin{align}\label{condition2ofdelta0}
	C(\omega,\lambda,\epsilon,\chi)\frac{\delta_0}{s^{1+3p}(1-s)} < \delta.
\end{align}
Since $\delta = C(\omega,\chi)s$, we may take
\[
\delta_0 = C(\omega,\lambda,\epsilon,\chi)s^{2+3p}(1-s)
\]
for some constant \(C(\omega,\lambda,\epsilon,\chi)\).

Hence, for any twisted Calabi flow starting from
$B_{\delta_0}^{2,\alpha}$, we have
\[
\varphi(t) \in \mathcal{V}_{\delta,\lambda}^{3,\alpha},
\quad \forall\, t \in [I,2I].
\]

For the twisted Calabi flow $\varphi(t)$ starting from any $\varphi_0\in B_{\delta_0}^{2,\alpha}$, we define
\[
T^s = \sup\bigl\{
\sigma \ge 2I : \varphi(t) \text{ exists on }  [0,\sigma] \text{ and }
\varphi(t) \in \mathcal{V}_{\delta,\lambda}^{3,\alpha},
\ \forall\, t \in [I,\sigma]
\bigr\}.
\]
We claim that \(T^s = +\infty\). Suppose to the contrary that
\(T^s < +\infty\).
Since
\(\varphi(T^s - I) \in \mathcal{V}_{\delta,\lambda}^{3,\alpha}
\subset B_\delta^{2,\alpha}\),
the short-time existence theorem together with the condition
\eqref{conditionofdelta} for $\delta$ ensures that the solution
extends to the interval \([T^s - I,\, T^s + I]\).
By the choice of $\delta$, we further have
\[
\varphi(t) \in \mathcal{V}_{\epsilon,\lambda}^{2},
\quad \forall\, t \in [T^s - I,\, T^s + I],
\qquad
\text{and} \qquad
\varphi(t) \in \mathcal{V}_{\epsilon,\lambda}^{4},
\quad \forall\, t \in [T^s,\, T^s + I].
\]
Hence \eqref{bootstrap-bound} still holds for all
\(t \in [I,\, T^s + I]\).
Since the a priori estimate \eqref{bootstrap-bound} is uniform in time, it follows that
\[
\|\varphi(t)\|_{c^{3,\alpha}(M)} < \delta,
\qquad \forall\, t \in [I,\, T^s + I],
\]
which contradicts the definition of \(T^s\).
Therefore, \(T^s = +\infty\).
	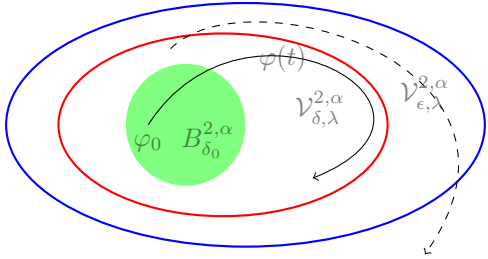
\begin{figure}[h]
	\centering
	\begin{tikzpicture}
		\pgfsetfillopacity{0.5}
		\fill[green](0,0) circle (0.8);e
		\node at (0.3,-0.2) {$B^{2,\alpha}_{\delta_0}$};
		\node at (3.2, 0.4) {$\mV^{2,\alpha}_{\epsilon, \lambda}$};
		\node at (1.8, 0.2) {$\mV^{2,\alpha}_{\delta, \lambda}$};
		\draw[red, thick] (0.5,0) ellipse (2.2cm and 1.2cm);
		\draw[blue, thick] (0.8,0)ellipse (3.2cm and 1.6 cm);
		\draw[->] (-0.5,0)node[anchor=north]{$\varphi_0$} .. controls (0.8,2) and (4,0.2) .. (1.7,-0.7);
		\draw[dashed,->] (-0.2,1).. controls (0.8,2) and (5,0.8) .. (3.2,-1.7);
		\node at (1.3,0.88) {$\varphi(t)$};
	\end{tikzpicture}
	\caption{Estimate \eqref{bootstrap-bound} guarantees that the twisted Calabi flow $\varphi(t)$ starting from $B_{\delta_0}^{2,\alpha}$ remains in $\mathcal{V}^{3,\alpha}_{\delta,\lambda}$ for all $t \ge I$.
		The dashed curve illustrates that a twisted Calabi flow starting from $\mathcal{V}^{2,\alpha}_{\delta,\lambda}$ stays within $\mathcal{V}^{2,\alpha}_{\epsilon,\lambda}$ for $t \in [0,2I]$, but may leave $\mathcal{V}^{2,\alpha}_{\epsilon,\lambda}$ after $2I$.
	}
\end{figure}

\eqref{L1-bound} implies that for $t>t'$,
\begin{align}\label{L1decay}
	\|\varphi(t)-\varphi(t')\|_{L^1}\leq
	\frac{2C}{\lambda_0}\, \mC^s(\varphi(t'))^{1/2}\to 0\quad\text{as}\quad t'\to \infty.
\end{align}
Thus $\varphi(t)\to\varphi(\infty)$ in the $L^1$ sense. On the other hand, short-time existence theorem implies that for any $t>I$
\begin{align*}
	\|\varphi(t+I )\|_{C^{4,\alpha}(M)}\leq CI^{-\frac{1}{2}}\|\varphi(t)\|_{C^{2,\alpha}(M)}\leq  CI^{-\frac{1}{2}}\delta.
\end{align*}
Thus $\|\varphi(t)\|_{C^{4,\alpha}(M)}$ is uniformly bounded and $\varphi(t)\to \varphi(\infty)$ in $C^{4,\alpha'}$ by sequence for any $0<\alpha'<\alpha$. The twisted Calabi energy of $\varphi(\infty)$ is 0 implies that $\varphi(\infty)$ defines a twisted cscK metric. By the uniqueness theorem for twisted cscK metrics \cite[Theorem 4.5]{MR3671939}, the limiting potential $\varphi(\infty)=0$. Moreover, \eqref{L1decay} implies $\|\varphi(t)\|_{L^1(M)}$ decays exponentially.

Similar to \eqref{bootstrap-bound}, for $t>2I$, we have
\begin{align*}
	\|\varphi(t)\|_{C^{3,\alpha}}
	\le C\!\left(
	\frac{1-s}{s}\mC^s(\varphi(\frac{t}{2}))^{\frac{1}{p}}e^{-\frac{\lambda_0 t}{p}}
	+\frac{1}{s^{3p}}\!\left(
	\|\varphi(\frac{t}{2})\|_{L^1}
	+\frac{2}{\lambda_0}\mC^s(\varphi(\frac{t}{2}))^{\frac{1}{2}}
	\right)
	\right).
\end{align*}
Thus $\varphi\to 0$ in $C^{3,\alpha}(M)$ as $t\to\infty$ since the right hand side of the above inequality tends to 0. Since the Calabi energy $\mC(\varphi(\frac{t}{2}))$ and $\|\varphi(\frac{t}{2})\|_{L^1(M)}$ decay exponentially, $\|\varphi(t)\|_{C^{3,\alpha}(M)}$ converges to 0 exponentially.

By the short-time existence, if choose $\varphi(t)$ as a initial value of the twisted Calabi flow, we have
\begin{align*}
\|\varphi(I+t)\|_{C^{5,\alpha}(M)}\leq CI^{-\frac{1}{2}}\|\varphi(t)\|_{C^{3,\alpha}(M)},
\end{align*}
$\|\varphi(t)\|_{C^{3,\alpha}(M)}\to 0$ exponentially as $t \to\infty$ implies $\|\varphi(t)\|_{C^{5,\alpha}(M)}\to 0$ exponentially.
Similarly,  $\|\varphi\|_{C^{k,\alpha}(M)}\to 0$ as $t\to\infty$ exponentially. Hence $\varphi(t)\to 0$ in smooth sense.


		%
		%
		%
		%
		%
		%

\end{proof}

\begin{rem}
	The stability neighborhood $B^{2,\alpha}_{\delta_0}$ depends smoothly on the twisted weight $s \in (0,1)$.
\end{rem}

\begin{rem}
	Unlike the classical Calabi flow near a cscK metric, where convergence may be to a different metric in the same automorphism orbit, the uniqueness theorem for twisted cscK metrics ensures convergence to the original metric $\omega$.
\end{rem}


The stability theorem has several important consequences for the geometry of twisted cscK metrics.


\begin{cor}[Dynamic stability]
	The twisted Mabuchi K-energy is dynamically stable near twisted cscK metrics. More precisely, twisted cscK metrics are local minimizers of $\mM^s$, and the twisted Calabi flow provides a gradient descent to these minimizers.
\end{cor}
%


\appendix

\section{Technical results and auxiliary estimates}
\label{sec:appendix}

This appendix collects various technical results and auxiliary estimates that are essential for the analysis in the main text but would disrupt the flow of the principal arguments. We provide complete proofs and detailed discussions of these foundational results.

\subsection{Interpolation and embedding results}

We begin with a refined interpolation inequality that plays a crucial role in the stability analysis of Section 5.

\begin{prop}[Gagliardo-Nirenberg type inequality] \label{prop:GN-inequality}
	For any $\varphi\in C^{3,\alpha}(M)$ and $p>1$, we have the interpolation inequality:
	\begin{align} \label{eq:GN-inequality}
		\|\varphi\|_{W^{2,p}(M,g)} \leq C \|\varphi\|_{C^{3,\alpha}(M)}^{\frac{3p-1}{3p}} \|\varphi\|_{L^1(M,g)}^{\frac{1}{3p}},
	\end{align}
	where the constant $C = C(m,p) > 0$ depends only on the dimension $m$ and the exponent $p$.
\end{prop}

\begin{proof}
		Since $M$ is compact and $\varphi \in C^{3,\alpha}(M)$, we have the trivial bound:
	\begin{align} \label{eq:W3p-bound}
		\|\varphi\|_{W^{3,p}(M,g)} \leq C_1 \|\varphi\|_{C^{3,\alpha}(M)},
	\end{align}
	where $C_1$ depends on $m$, $p$, and the geometry of $(M,g)$.

	We use the Gagliardo-Nirenberg interpolation.
	Gagliardo-Nirenberg inequality originated from the work of Gagliardo (\cite{MR109940}) and Nirenberg (\cite{MR102740}). We use its modern form(see \cite{MR3813967}): for $1 \leq p,p_1, p_2\leq \infty < k$ and $s, s_1, s_2\geq 0$ satisfying:
	\[
	s_1\leq s_2, \quad s=\theta s_1+(1-\theta)s_2,\quad \frac{1}{p} = \frac{\theta}{p_1}   + \frac{1-\theta}{p_2},
	\]
	we have:
	\[
	\|\varphi\|_{W^{s,p}(M,g)} \leq C \|\varphi\|_{W^{s_1,p_1}(M,g)}^\theta \|\varphi\|_{W^{s_2,p_2}(M,g)}^{1-\theta}.
	\]
	for any $\varphi\in W^{s_1, p_1}(M)\cap W^{s_2, p_2}(M)$.
	Applying this with $s=2,s_1=0,s_2=3$, $\theta=1/3$, $p=p_1=p_2$, we obtain:
	\begin{align} \label{eq:GN-application}
		\|\varphi\|_{W^{2,p}(M,g)} \leq C_2 \|\varphi\|_{W^{3,p}(M,g)}^{2/3} \|\varphi\|_{L^p(M,g)}^{1/3}.
	\end{align}
	Using the $L^\infty$ bound and H\"older's inequality, we have
	\begin{align*}
		\|\varphi\|_{L^p(M,g)}
		&= \left( \int_M |\varphi|^p \, \omega^m  \right)^{1/p} \\
		&\leq \left( \int_M |\varphi| \cdot \|\varphi\|_{L^\infty}^{p-1} \,\omega^m  \right)^{1/p} \\
		&= \|\varphi\|_{L^\infty(M)}^{(p-1)/p} \|\varphi\|_{L^1(M,g)}^{1/p} \\
		&\leq \|\varphi\|_{C^{3,\alpha}(M)}^{(p-1)/p} \|\varphi\|_{L^1(M,g)}^{1/p}. \label{eq:Lp-L1-bound}
	\end{align*}
Substituting \eqref{eq:W3p-bound} and \eqref{eq:Lp-L1-bound} into \eqref{eq:GN-application}:
	\begin{align*}
		\|\varphi\|_{W^{2,p}(M,g)}
		&\leq C_2 \left( C_1 \|\varphi\|_{C^{3,\alpha}(M)} \right)^{2/3} \left( \|\varphi\|_{C^{3,\alpha}(M)}^{(p-1)/p} \|\varphi\|_{L^1(M,g)}^{1/p} \right)^{1/3} \\
		&= C_1^{2/3} C_2 \|\varphi\|_{C^{3,\alpha}(M)}^{\frac{2}{3} + \frac{p-1}{3p}} \|\varphi\|_{L^1(M,g)}^{\frac{1}{3p}} \\
		&= C \|\varphi\|_{C^{3,\alpha}(M)}^{\frac{3p-1}{3p}} \|\varphi\|_{L^1(M,g)}^{\frac{1}{3p}},
	\end{align*}
	which completes the proof.
\end{proof}


\subsection{Metric equivalence and norm comparisons}

The following result establishes the equivalence of Sobolev norms under metric perturbations, which is fundamental for our analysis of the twisted Calabi flow.

\begin{prop}[Norm equivalence under metric perturbations] \label{prop:norm-equivalence}
	Suppose $\omega_\varphi$ is a K\"ahler metric satisfying $\lambda\omega \leq \omega_\varphi \leq \lambda^{-1}\omega$ for some $\lambda > 0$, and $\|\varphi\|_{C^{3,\alpha}(M)} < \epsilon_0$ for sufficiently small $\epsilon_0$. Then for any $f \in W^{2,p}(M,g)$ with $1 < p < \infty$, we have:
	\begin{align}
		C^{-1} \|f\|_{W^{2,p}(M,g)} \leq \|f\|_{W^{2,p}(M,g_\varphi)} \leq C \|f\|_{W^{2,p}(M,g)},
	\end{align}
	where $C = C(\lambda, p, \epsilon_0, m) > 0$.
\end{prop}
\begin{proof}
	We prove the two inequalities separately, analyzing how each term in the $W^{2,p}$ norm transforms under the metric change.

	The metric equivalence implies:
	\begin{align} \label{eq:volume-comparison}
		\lambda^m \omega^m  \leq \omega^m_\varphi \leq \lambda^{-m} \omega^m .
	\end{align}
	This immediately gives the $L^p$ norm equivalence:
	\[
	\lambda^{m/p} \|f\|_{L^p(M,g)} \leq \|f\|_{L^p(M,g_\varphi)} \leq \lambda^{-m/p} \|f\|_{L^p(M,g)}.
	\]
	For the gradient term, we have:
	\[
	|\nabla_\varphi f|_\varphi^2 = g_\varphi^{i\bar{j}} \partial_i f \partial_{\bar{j}} f.
	\]
	It follows
	\begin{align*}
		\lambda^{\frac{p}{2}}|\nabla u|_{g }^{p}\leq|\nabla_\varphi u|^p_\varphi\leq \lambda^{-\frac{p}{2}}|\nabla u|_{g }^{p}
	\end{align*}
	and
	\begin{align} \label{eq:gradient-comparison}
		\lambda^{\frac{m}{p}+\frac{1}{2}} \|\nabla f\|_{L^p(M,g)} \leq \|\nabla_\varphi f\|_{L^p(M,g_\varphi)} \leq \lambda^{-\frac{m}{p}-\frac{1}{2}}\|\nabla f\|_{L^p(M,g)}.
	\end{align}

The Hessians with respect to different metrics are related by:
\begin{align*}
\nabla_\alpha\nabla_{\bar\beta}f=\nabla^\varphi_\alpha\nabla^\varphi_{\bar\beta}f=\partial_\alpha\partial_{\bar \beta}f
\end{align*}
and
	\begin{align*}
		\nabla^\varphi_\alpha \nabla^\varphi_\beta f &= \nabla_\alpha \nabla_\beta f - \Gamma^{\varphi,\xi}_{\alpha\beta} \nabla_\xi f + \Gamma^\xi_{\alpha\beta} \nabla_\xi f \\
		&= \nabla_\alpha \nabla_\beta f + (\Gamma^\xi_{\alpha\beta} - \Gamma^{\varphi,\xi}_{\alpha\beta}) \nabla_\xi f.
	\end{align*}
	The Christoffel symbol difference can be computed as:
	\begin{align*}
		\Gamma^{\varphi,\xi}_{\alpha\beta} - \Gamma^\xi_{\alpha\beta}
		&= g_\varphi^{\xi\bar{\eta}} (\partial_\alpha g_{\beta\bar{\eta}} + \varphi_{\alpha\beta\bar{\eta}}) - g^{\xi\bar{\eta}} \partial_\alpha g_{\beta\bar{\eta}} \\
		&= (g_\varphi^{\xi\bar{\eta}} - g^{\xi\bar{\eta}}) \partial_\alpha g_{\beta\bar{\eta}} + g_\varphi^{\xi\bar{\eta}} \varphi_{\alpha\beta\bar{\eta}}.
	\end{align*}
	Using the identity $g_\varphi^{\xi\bar{\eta}} - g^{\xi\bar{\eta}} = -g^{\xi\bar{\delta}} \varphi_{\gamma\bar{\delta}} g_\varphi^{\gamma\bar{\eta}}$ and the bound $\|\varphi\|_{C^{3,\alpha}} < \epsilon_0$, we obtain:
	\[
	|\Gamma^{\varphi,\xi}_{\alpha\beta} - \Gamma^\xi_{\alpha\beta}| \leq C\epsilon_0.
	\]
	Therefore:
	\begin{align*}
		|\nabla^\varphi \nabla^\varphi f - \nabla\nabla f|_g \leq C\epsilon_0 |\nabla f|_g.
	\end{align*}
	Combining with the volume comparison \eqref{eq:volume-comparison}:
	\begin{align} \label{eq:hessian-comparison}
		\|\nabla^\varphi \nabla^\varphi f\|_{L^p(M,g_\varphi)}
		&\leq \|\nabla\nabla f\|_{L^p(M,g_\varphi)} + C\epsilon_0 \|\nabla f\|_{L^p(M,g_\varphi)} \\
		&\leq \lambda^{-m/p} \|\nabla\nabla f\|_{L^p(M,g)} + C\epsilon_0 c_2 \|\nabla f\|_{L^p(M,g)}.
	\end{align}
	The reverse inequality follows similarly.

	Combining \eqref{eq:volume-comparison}, \eqref{eq:gradient-comparison}, and \eqref{eq:hessian-comparison}, and choosing $\epsilon_0$ sufficiently small, we obtain the desired norm equivalence.
\end{proof}

\begin{rem}
	This result is crucial for transferring elliptic estimates between different background metrics. The smallness assumption on $\|\varphi\|_{C^{3,\alpha}}$ ensures that the metric perturbation is sufficiently regular to preserve the elliptic structure.
\end{rem}

\bigskip

\noindent {\it\bf{Acknowledgements}}:
This work was completed during  Jie He's visit to Haozhao Li at the University of Science and Technology of China(USTC). This visit was supported by the Mathematics Tianyuan Fund of National Natural Science Foundation of China   No. 12426674 and No. 12426669.
Jie He wishes to thank the visiting scholar program and USTC for its hospitality.

\medskip
\noindent {\it\bf{Conflict of Interest}}: The authors declare no conflict of interest.

\medskip
\noindent {\it\bf{Data availability}}:
No data was used for the research described in the article.

\bibliographystyle{amsplain}
\providecommand{\bysame}{\leavevmode\hbox to3em{\hrulefill}\thinspace}
\providecommand{\MR}{\relax\ifhmode\unskip\space\fi MR }
\providecommand{\MRhref}[2]{%
	\href{http://www.ams.org/mathscinet-getitem?mr=#1}{#2}
}
\providecommand{\href}[2]{#2}

\end{document}